\documentclass[a4paper,11pt]{article}
\usepackage[latin1]{inputenc}
\usepackage{amsfonts}
\usepackage{amsthm}
\usepackage{amsmath}
\usepackage{amsfonts}
\usepackage{latexsym}
\usepackage{amssymb}
\usepackage{dsfont}

\newtheorem{fed}{Definition}[section]
\newtheorem{teo}[fed]{Theorem}
\newtheorem*{teo*}{Theorem}
\newtheorem{lem}[fed]{Lemma}
\newtheorem{cor}[fed]{Corollary}
\newtheorem{pro}[fed]{Proposition}
\theoremstyle{definition}
\newtheorem{rem}[fed]{Remark}

\newtheorem{exa}[fed]{Example}

\newtheorem{num}[fed]{}

\def\hn{\nu (\cF_0 \coma \ca)}
\def\hr{r (\cF_0 \coma \ca)}
\def\hm{\mu (\cF_0 \coma \ca)}
\def\hc{c (\cF_0 \coma \ca)}
\newcommand{\IN}[1]{\mathbb {I} _{#1}}
\def\In{\mathbb {I} _n}
\def\IM{\mathbb {I} _m}
\def\suml{\sum\limits}

\def\bce{\begin{center}}
\def\ece{\end{center}}
\newcommand{\trivial}{\{0\}}
\DeclareMathOperator{\FP}{FP\,}

\def\cD{\mathcal D}

\def\subim{_{i\in \IN{n}}\,}
\def\py{\peso{and}}
\def\rk{\text{\rm rk}}
\def\noi{\noindent}
\def\cF{\mathcal F}
\def\QED{\hfill $\square$}
\def\EOE{\hfill $\triangle$}
\def\EOEP{\tag*{\EOE}}
\newcommand{\peso}[1]{ \quad \text{ #1 } \quad }
\def\uno{\mathds{1}}
\def\bm{\left[\begin{array}}
\def\em{\end{array}\right]}
\def\ben{\begin{enumerate}}
\def\een{\end{enumerate}}
\def\bit{\begin{itemize}}
\def\eit{\end{itemize}}
\def\barr{\begin{array}}
\def\earr{\end{array}}
\def\igdef{\ \stackrel{\mbox{\tiny{def}}}{=}\ }

\def\k{n}

\def\la{\lambda}
\def\al{\alpha}
\def\N{\mathbb{N}}
\def\R{\mathbb{R}}
\def\C{\mathbb{C}}
\def\cA{\mathcal{A}}
\def\cI{\mathcal{I}}
\def\cC{\mathcal{C}}

\def\cH{\mathcal{H}}
\def\cK{\mathcal{K}}

\def\cS{{\cal S}}

\def\cM{{\cal M}}
\def\cB{{\cal B}}

\def\cV{{\cal F}}
\def\cU{{\cal U}}
\def\cW{{\cal G}}
\def\ca{\mathbf{a}}
\def\cb{\mathbf{b}}

\def\ese{\mathcal{S}}

\def\eme{\mathcal{M}}
\def\ene{\mathcal{N}}

\def\vacio{\varnothing}
\def\orto{^\perp}
\def\inc{\subseteq}

\def\sii{ if and only if }
\def\inv{^{-1}}
\def\rai{^{1/2}}

\def\api{\langle}
\def\cpi{\rangle}
\def\vlm{\nu_{\la\coma m}}

 \DeclareMathOperator{\tr}{tr}
\DeclareMathOperator{\gen}{span}

\DeclareMathOperator{\leqp}{\leqslant}
\DeclareMathOperator{\geqp}{\geqslant}

\def\RS{\mathbf{F} }
\def\Fnd{\RS (n \coma d)}

\def\RSV{\cF= \{f_i\}_{i\in \, \IN{n}}}

\def\RSW{\cW= \{g_i\}_{i\in \, \IN{n}}}

\newcommand{\hil}{\mathcal{H}}
\newcommand{\op}{L(\mathcal{H})}
\newcommand{\lhk}{L(\mathcal{H} \coma \mathcal{K})}
\newcommand{\lhcn}{L(\mathcal{H} \coma \C^n)}
\newcommand{\lkh}{L(\mathcal{K} \coma \mathcal{H})}

\newcommand{\posop}{L(\mathcal{H})^+}

\def\H{{\cal H}}
\def\glh{\mathcal{G}\textit{l}\,(\cH)}

\newcommand{\cene}{\mathbb{C}^n}
\newcommand{\mat}{\mathcal{M}_d(\mathbb{C})}
\newcommand{\matn}{\mathcal{M}_n(\mathbb{C})}

\newcommand{\matsa}{\mathcal{H}(n)}

\newcommand{\matu}{\mathcal{U}(n)}

\newcommand{\matpos}{\mat^+}
\newcommand{\matposn}{\matn^+}

\newcommand{\matinv}{\mathcal{G}\textit{l}\,(n)}
\newcommand{\matinvd}{\mathcal{G}\textit{l}\,(d)}
\def\gld{\matinvd^+}

\newcommand{\matrec}[1]{\mathcal{M}_{#1} (\mathbb{C})}

\def\beq{\begin{equation}}
\def\eeq{\end{equation}}

\def\pausa{\medskip\noi}

\begin{document}

\title{ {\bf Optimal dual frames and frame completions for majorization}}
\author{Pedro G. Massey, Mariano A. Ruiz  and Demetrio Stojanoff\thanks{Partially supported by CONICET 
(PIP 5272/05) and  Universidad de La PLata (UNLP 11 X472).} }
\author{P. G. Massey, M. A. Ruiz and D. Stojanoff \\ {\small Depto. de Matem\'atica, FCE-UNLP,  La Plata, Argentina
and IAM-CONICET  
}}
\date{}
\maketitle

\centerline{Dedicated to the memory of ``el flaco" L. A. Spinetta.}

\begin{abstract}  In this paper we consider two problems in frame theory. On the one hand,
 given a set of vectors $\mathcal F$ we describe the spectral and geometrical structure of optimal completions of $\mathcal F$ by a finite family of vectors with prescribed norms, where optimality is measured with respect to majorization. In particular, these optimal completions are the minimizers of a family of convex functionals that include the mean square error and the Benedetto-Fickus' frame potential.  On the other hand, given a fixed frame $\mathcal F$ we describe explicitly the spectral and geometrical structure of optimal frames $\mathcal G$ that are in duality with $\mathcal F$ and such that the Frobenius norms of their analysis operators is bounded from below by a fixed constant. In this case, optimality is measured with respect to submajorization of the frames operators. Our approach relies on the description of the spectral and geometrical structure of matrices that minimize submajorization on sets that are naturally associated with the problems above.
\end{abstract}

\def\coma{\, , \, }



\tableofcontents

\section{Introduction}

Finite frame theory is a well established research field that has attracted the attention of many researchers (see \cite{TaF,Chr,HLmem} for general references to frame theory). On the one hand, finite frames provide  redundant linear encoding-decoding schemes that are useful when dealing with transmission of signals through noisy channels . Indeed, the redundancy of frames allows  for reconstruction of a signal, even when some frame coefficients are lost. Moreover, frames have also shown to be robust under erasures of the frame coefficients when a blind reconstruction strategy is considered (see \cite{Bod,Pau,BodPau,HolPau,LeHanagre,LoHanagre,MRS3}). On the other hand, there are several problems in frame theory that have deep relations with problems in other areas of mathematics (such as matrix analysis, operator theory and operator algebras) which constitute a strong motivation for research. For example, we can mention the relation between the Feichtinger conjecture in frame theory and some major open problems in operator algebra theory such as the Kadison-Singer problem (see \cite{CCL,CFTW}). Other examples of this phenomenon are
the design problem in frame theory, the so-called Paulsen problem in frame theory and frame completion problems (\cite{Illi,CFMP,Casagregado,CC,CMTL,DHST,DFKLOW,FWW,FicMixPot,KLagregado,MR0}) which are known to be equivalent to different aspects of the Schur-Horn theorem. Recently, matrix analysis has served as a tool to show some structural properties of minimizers of the Benedetto-Fickus frame potential (\cite{BF,Phys}) and other convex functionals in the finite setting (\cite{MR,MRS,MRS2}).

Following \cite{Illi,DHST,MR0,MR,MRS,MRS2}, in this paper we explore  new connections of problems that arise naturally in frame theory with some results in matrix theory related with the notion of (sub)majorization between vectors and positive matrices. Indeed, one of the main problems in frame theory is the design of frames with some prescribed parameters and such that they are optimal in some sense. Optimal frames $\cF$ are usually the minimizers of a tracial convex functional i.e., a functional of the form $P_f(\cF)=\tr(f(S_\cF))$ for some convex function $f(x)$, where $S_\cF$ 
is the frame operator of $\cF$. For example, we mention the Benedetto-Fickus' frame potential 
(i.e. $f(x)=x^2$) or the mean square error (i.e. $f(x)=x^{-1}$) or the negative of von Neumann's entropy 
(i.e. $f(x)=x\,\log(x)$). Thus, in many situations it is natural to ask whether the optimal frames corresponding to different convex potentials coincide: that is, whether optimality with respect to these potentials is an structural property. One powerful tool to deal with this type of problems is the notion of (sub)-majorization between positive operators, because of its relation with tracial inequalities with respect to convex functions as above (see Section \ref{subsec 2.2}). 
Hence, a (sub)-majorization based strategy can reveal structural properties of optimal frames. It is worth pointing out that (sub)-majorization is not a total preorder and therefore the task of computing minimizers of this relation within a given set of positive operators - if such minimizers exist - is usually a non trivial problem.

In this paper we consider the following two optimality 
problems in frame theory in terms of (sub)-majorization (see Section \ref{Pre} for the notation and terminology).  Given a finite sequence of vectors $\cF_0 \inc \hil\cong\mathbb C^d$ and a 
finite sequence of positive numbers $\cb$ we are interested in 
computing optimal frame completions of $\cF_0$, denoted by $\cF$, 
obtained by adding vectors with norms prescribed by the entries 
of $\cb$ (see Section \ref{Fick} for the motivation and a detailed description of this problem). In this context we show the existence of minimizers of majorization in the set of frame completions of $\cF_0$ with prescribed norms, under certain hypothesis on $\cb$; we also compute the spectral and geometrical structure of these optimal completions. Our results can be considered as a further step in the classical frame completion and frame design problems considered in \cite{Illi,CFMP,Casagregado,CMTL,DHST,DFKLOW,FicMixPot,KLagregado}. In particular, we solve the frame completion problem recently posed in \cite{FicMixPot}, where optimality is measured with respect to 
the mean square error of the completed frame.

On the other hand, given a fixed frame $\cF$ for a finite dimensional 
Hilbert space $\hil\cong\mathbb C^d$,  
let 
$\cD(\cF)$ denote the set of all frames $\cW$ that are in duality with $\cF$. 
It is well known that the canonical dual of $\cF$, denoted $\cF^\#$, has some optimality properties among the elements in $\cD(\cF)$. Nevertheless, although optimal in some senses, there might be alternate duals that are more suitable for applications (see \cite{BLagreg,Han,LeHanagre,LoHanagre,MRS3,WES}).  In order to search for optimal alternative duals for $\cF$ we restrict attention the set $\cD_t(\cF)$ which consists of frames $\cW$ that are in duality with $\cF$ and such that the Frobenius norm of their frame operators is bounded from below by a constant $t$. 
Therefore, in this paper we show the existence of minimizers of submajorization in $\cD_t(\cF)$ and we 
explicitly describe their spectral and geometrical structure (see Section \ref{duales con rest} for the motivation and a detailed description of this problem).

Both problems above are related with the minimizers of (sub)majorization in certain sets $\mathcal S$ of positive semidefinite matrices that arise naturally. We show that these sets $\mathcal S$ that we consider have minimal elements with respect to (sub)-majorization, a fact that is of independent interest (see Theorems \ref{prop consec de la defi}, \ref{tutti nu} and \ref{el St general}). 
Notably, the existence of such minimizers is essentially obtained with insights coming from frame theory.

The paper is organized as follows: In Section \ref{Pre} we establish 
the notation and terminology used throughout the paper, and we state some 
basic facts from frame theory and majorization theory. 
In Sections \ref{Fick} and \ref{duales con rest} we give a detailed description of the two main problems of frame theory mentioned above, including motivations, related results and specific notations. Section \ref{description of main prob} 
ends with
the definitions and statements of the matrix theory results of the paper, which 
give a unified matrix model for the frame problems; in order to avoid some technical aspects of these results, 
their proofs are presented in an Appendix (Section \ref{appendixiti}).
In Section \ref{la solucion frames}  we 
apply the previous analysis of the matrix model to obtain the solutions
of the frame problems, including algorithmic implementations and several examples. 
With respect to the problem of optimal completions, 
we obtain a complete description in several cases, that include the case 
of uniform norms for the added vectors.
With respect to the problem of minimal duals, 
we completely describe their spectral and geometrical structure. 
The Appendix, Section \ref{appendixiti}, contains the proofs
of the matrix theory results of Section \ref{unified}; 
it is divided in three subsections in which we develop the following steps: 
the characterization of the set of vectors of eigenvalues of elements in the matrix model,
the description of the minimizers for sub-majorization in this set, and 
the description of the geometric structure of the matrices which are minimizers 
for sub-majorization in the matrix model.

\section{Preliminaries}\label{Pre}

In this section we describe the basic notions that we shall consider throughout the paper. We first establish the general notations and then we recall the basic facts from frame theory that are related with our main results. 
 Finally, we  describe submajorization which is a notion from matrix analysis, that will play a major role in this note.

\subsection{General notations.}
Given $m \in \N$ we denote by $\IM = \{1, \dots , m\} \inc \N$ and 
$\uno = \uno_m  \in \R^m$ denotes the vector with all its entries equal to $1$. 
For a vector $x\in \R^m$ we denote by $x^\downarrow$ the rearrangement
of $x$ in  decreasing order, and $\R^m \,^\downarrow = \{ x\in \R^m : x = x^\downarrow\}$
the set of ordered vectors. 

\pausa
Given $\cH \cong \C^d$  and $\cK \cong \C^n$, we denote by $\lhk $ 
the space of linear operators $T : \cH \to \cK$. 
Given an operator $T \in \lhk$, $R(T) \inc \cK$ denotes the
image of $T$, $\ker T\inc \cH$ the null space of $T$ and $T^*\in \lkh$ 
the adjoint of $T$. If $d\le n$ we say that $U\in \lhk$ is an isometry 
if $U^*U = I_\cH\,$. In this case, $U^*$ is called a coisometry. 
If $\cK = \cH$ we denote by $\op = L(\cH \coma \cH)$, 
by $\glh$ the group of all invertible operators in $\op$, 
 by $\posop $ the cone of positive operators and by
$\glh^+ = \glh \cap \posop$. 
If $T\in \op$, we  denote by   
$\sigma (T)$ the spectrum of $T$, by $\rk\, T= \dim R(T) $  the rank of $T$,
and by $\tr T$ the trace of $T$. By fixing an orthonormal basis (onb) 
of the Hilbert spaces involved, we shall identify operators with 
matrices, using the following notations:

\pausa 
By $\matrec{n,d} \cong L(\C^d \coma \C^n)$ we denote the space of complex $n\times d$ matrices. 
If $n=d$ we write $\matn = \matrec{n,n}$.  
$\matsa$ is the $\R$-subspace of selfadjoint matrices,  
$\matinv$ the group of all invertible elements of $\matn$, $\matu$ the group 
of unitary matrices, 
$\matposn$ the set of positive semidefinite
matrices, and $\matinv^+ = \matposn \cap \matinv$. 
If $d\le n$, we denote by $\cI(d\coma n) \inc 
\matrec{n\coma d}$ the set of isometries, i.e. those 
$U\in \matrec{n\coma d} $ such that $U^*U = I_d\,$. 
Given $S\in \matrec{n}^+$, we write $\la(S) \in 
\R_{+}^n\,^\downarrow$ the 
vector of eigenvalues of $S$ - counting multiplicities - arranged in decreasing order. 
If $ \lambda(S)=\la= (\lambda_1 \coma \ldots \coma \lambda_n) \in \R_{+}^n\,^\downarrow\,$, a system 
 $\{h_i\}_{i\in \IN{n}} \inc \C^n$ is a 
``ONB of eigenvectors for $S\coma \la \,$" if it is an  
orthonormal basis for $\C^n$ such that 
$S\,h_i=\lambda_i\,h_i$ for every $i\in \IN{n}\,$.

\pausa
If $W\inc \cH$ is a subspace we denote by $P_W \in \posop$ the orthogonal 
projection onto $W$, i.e. $R(P_W) = W$ and $\ker \, P_W = W^\perp$. 
Given $x\coma y \in \cH$ we denote by $x\otimes y \in \op$ the rank one 
operator given by 
$x\otimes y \, (z) = \api z\coma y\cpi \, x$ for every $z\in \cH$. Note that
if $\|x\|=1$ then $x\otimes x = P_{\gen\{x\}}\,$. 

\pausa
For vectors in $\cene$ we shall use the euclidean norm.  
On the other hand,  for  $T\in \matrec{n\coma d}$ we shall use both the spectral norm, denoted $\|T\|$, and the Frobenius norm, denoted $\|T\|_{_2}$, given by
$$\|T\| =  \max\limits_{\|x\|=1}\|Tx\| \peso{ and }  \|T\|_{_2} = (\tr \, T^*T )\rai = 
\big( \, \suml_{i \in \In\, , \ j\in \IN{d}  } |T_{ij}|^2 \, \big)\rai \ . $$
\subsection{Basic framework of finite frames and their dual frames}\label{basic}

In what follows we consider $(\k,d)$-frames. See 
\cite{BF,TaF,Chr,HLmem,MR} for detailed expositions of several aspects of this notion. 

\pausa
Let $d, n \in \N$, with $d\le n$. Fix a Hilbert space $\hil\cong \C^d$. 
A family $ \RSV \in  \cH^n $  is an 
$(\k,d)$-frame for $\cH$  if there exist constants $A,B>0$ such that
\beq\label{frame defi} A\|x\|^2\leq \sum_{i=1}^n |\left \langle x \, , f_i\right \rangle|^2\leq B \|x\|^2 \peso{for every} x\in \hil \ .
\eeq
The {\bf frame bounds}, denoted by  $A_\cF, B_\cF$ are the optimal constants in \eqref{frame defi}. If $A_\cF=B_\cF$ we call $\cF$ a tight frame.
Since $\dim \hil<\infty$, a family  $\RSV$  is an 
$(\k,d)$-frame 
 \sii $\gen\{f_i: i \in \In \} = \cH$.  
We shall denote by $\RS = \RS(n \coma d)$ the set of all $(\k,d)$-frames for $\cH$. 

\pausa
Given $\RSV \in  \cH^n $, the operator $T_\cV \in L(\hil\coma\C^n)$ defined
by 
 \beq \ T_\cV\, x= \big( \,\api x \coma f_i\cpi\,\big)  \subim 
\ , \peso{for every} x\in \cH \,
\eeq
is the {\bf analysis} operator of $\cF$.  Its adjoint $T_\cV^*$ is called the {\bf synthesis} operator: 
$$
T_\cV^* \in L(\C^n\coma \cH)  \peso{given by}
T_\cV ^* \, v =\sum_{i\in \, \IN{m}} v_i\, f_i \peso{for every}
v = (v_1\coma \dots\coma v_n)\in \C^n \ . 
$$
Finally, we define the {\bf frame operator} of $\cV$ as 
$S_\cV = T_\cV^*\  T_\cV = \sum_{i \in \In} f_i \otimes f_i 
\in \posop\, $. Notice that, if 
$\cF \in \Fnd$, then  $\api S_\cV \, x\coma x\cpi \, = \sum\subim \, 
 \big|\, \api x \coma f_i\cpi \, \big|^2$ for every $x\in \cH$, so  	
$S_\cF\in \glh^+$ and 
\beq\label{const RS}
A_\cV \, \|x\|^2 \, \le \, \api S_\cV \, x\coma x\cpi 
 \,  \le \, B_\cV \, \|x\|^2  \peso{for every} x\in \cH \ .
\eeq
In particular, $A_\cV   =\la_{\min} (S_\cV) = \|S_\cV\inv \| \inv$ and $ 
\la_{\max} (S_\cV) = \|S_\cV \| = B_\cV \,$.
Moreover, $\cV $ is tight if and only if $S_\cV = 
\frac{\tau}{d}  \, I_\H\,$, where $\tau = 
\tr S_\cV = \sum\subim \|f_i\|^2 \,$.

\pausa
The frame operator plays an important role in the reconstruction of a vector $x$ using its frame coefficients 
$\{\api  x\coma f_i\cpi \,\}_{i\in \IN{n}}$. This leads to the definition of the canonical dual frame associated to $\cF$:
for every $\RSV\in \RS(n \coma d)$, the {\bf canonical dual} frame associated to $\cV$ is the sequence 
$\cV^\#\in \RS$ defined by
$$
\cV^\# \igdef S_\cV ^{-1} \cdot \cV  =  \{S_\cV ^{-1} \,f_i\,\}_{i\in \, \IN{m}} \in \RS(n \coma d) \ .
$$
Therefore, we obtain the reconstruction formulas
\begin{equation}\label{ec recons}
x= \sum\subim  \api  x \coma f_i\cpi \, S_\cV^{-1} \,f_i 
= \sum\subim  \api  x \coma S_\cV^{-1} \, f_i\cpi \, f_i  \peso{for every} x\in \cH \ .
\end{equation} 
Observe that the canonical dual $\cV^\#$ satisfies that given $x\in \cH$, then 
\beq\label{SMP}
T_{\cV^\#}\, x= \big( \,\api x \coma  S_\cV\inv \, f_i\cpi\,\big)  \subim 
= \big( \,\api S_\cV\inv \, x \coma   f_i\cpi\,\big)  \subim 
\peso{for} x\in \cH 
\implies 
T_{\cV^\#}= T_\cV \, S_\cV\inv \ .
\eeq 
Hence $T_{\cV^\#}^* \, T_\cV = 
I_\cH $ and $S_{\cV^\#} =  S_{\cV}^{-1}\,T_\cV^*\  T_\cV\,S_{\cV}^{-1} = S_{\cV}^{-1}\,$.  

\pausa 
In their seminal work \cite{BF}, Benedetto and Fickus introduced a functional defined (on unit norm frames), the so-called frame potential, given by $$ \FP(\{f_i\}_{i\in \In} )=\sum_{i,\,j\,\in \In}|
\api f_i\coma f_j \cpi |\,^2\ .$$ One of their major results shows that tight unit norm frames - which form an important class of frames because of their simple reconstruction formulas - can be characterized as (local) minimizers of this functional among unit norm frames. Since then, there has been interest in (local) minimizers of the frame potential within certain classes of frames, since such minimizers can be considered as natural substitutes of tight frames (see \cite{Phys,MR,MRS}). Notice that, given $\cF=\{f_i\}_{i\in \IN{n}}\in  \cH^n$ then $\FP(\cF)=\tr(S_\cF^2)=\sum_{i\in \IN{d}}\lambda_i(S_\cF)^2$.
These remarks have motivated 
the definition of general convex potentials as follows:

\begin{fed}\label{pot generales}\rm
Let $f:[0,\infty)\rightarrow [0,\infty)$ be a 
convex function. Following \cite{MR} we consider the (generalized) frame potential associated to $f$, denoted $P_f$, given by
\begin{equation}
P_f(\cF)=\tr(f(S_\cF)) \peso {for} 
\cF=\{f_i\}_{i\in \IN{n}}\in  \cH^n \ .\EOEP
\end{equation}  
\end{fed}

\pausa
Of course, one of the most important generalized potential is the Benedetto-Fickus' (BF) frame potential. As shown in \cite[Sec. 4]{MR} these convex functionals (which are related with the so-called entropic measures of frames) share many properties with the BF-frame potential. Indeed, under certain restrictions both the spectral and geometric structures of minimizers of these potentials coincide (see \cite{MR}).

\subsection{Submajorization} \label{subsec 2.2}

Next we briefly describe submajorization, a notion from matrix analysis theory that will be used throughout the paper. For a detailed exposition of submajorization see \cite{Bat}.

\pausa 
 Given $x,\,y\in \R^d$ we say that $x$ is
{\bf submajorized} by $y$, and write $x\prec_w y$,  if
$$\suml_{i=1}^k x^\downarrow _i\leq \suml_{i=1}^k y^\downarrow _i \peso{for every} k\in \mathbb I_d \,.$$  
If $x\prec_w y$ and $\tr x = \suml_{i=1}^dx_i=\suml_{i=1}^d y_i = \tr y$,  then we say that $x$ is
{\bf majorized} by $y$, and write $x\prec y$. 

\pausa
On the other hand we write 
$x \leqp y$ if $x_i \le y_i$ for every $i\in \mathbb I_d \,$.  It is a standard  exercise 
to show that $x\leqp y \implies x^\downarrow\leqp y^\downarrow  \implies x\prec_w y $. 
Majorization is usually considered because of its relation with tracial inequalities 
for convex functions. 
Indeed, given $x,\,y\in \R^d$ and  $f:I\rightarrow \R$ a 
convex function defined on an interval $I\inc \R$ such that 
$x,\,y\in I^d$,  then (see for example \cite{Bat}): 
\ben 
\item If one assumes that $x\prec y$, then 
$ 
\tr f(x) \igdef\suml_{i=1}^df(x_i)\leq \suml_{i=1}^df(y_i)=\tr f(y)\ .
$
\item If only $x\prec_w y$,  but the map $f$ is also increasing, then  still 
$\tr f(x) \le \tr f(y)$. 
\item If $x\prec_w y$ and $f$ is an strictly convex function such that $\tr(f(x))=\tr(f(y))$ then there exists a permutation $\sigma$ of $\IN{d}$ such that $y_i=x_{\sigma(i)}$ for $i\in \IN{d}\,$. 
\een
The notion of submajorization can be extended to the context of self-adjoint matrices as follows: given $S_1,\,S_2\in \mathcal H(d)$ we say that $S_1$ is {\bf submajorized} by $S_2$, denoted $S_1\prec_w S_2$,  if $\lambda(S_1)\prec_w \lambda(S_2)$. If $S_1\prec_w S_2$ and $\tr(S_1)=\tr(S_2)$ we say that $S_1$ is {\bf majorized} by $S_2$ and write $S_1\prec S_2$. Thus, $S_1\prec S_2$ if and only if $\lambda(S_1)\prec \lambda(S_2)$. Notice that (sub)majorization is an spectral relation between self-adjoint operators.

\section{Description and modeling of the main problems} \label{description of main prob}

We begin this section with a detailed description of our two main problems together with their motivations. In both cases we search for optimal frame designs (frame completions and duals), that are of potential interest in applied situations. In order to tackle these problems we obtain (see Sections \ref{Fick} and \ref{duales con rest}) equivalent versions of them in a matrix analysis context. In section \ref{unified} we present a unified matrix model and develop some notions and results that allow us to solve the two problems in frame theory (see Section \ref{la solucion frames}).

\def\n0{n_{ \text{\rm \tiny o}}}
\subsection{Frame completions with prescribed norms}\label{Fick}

We begin by describing the following frame completion problem posed in \cite{FicMixPot}. 
Let $\cH \cong \C^d$ and let $\cF_0=\{f_i\}_{i\in \IN{\n0}}
\in \cH^{\n0  }$ be a fixed (finite) sequence of vectors. Let $n>\n0  \,$ be an integer; denote by $k = n-\n0  \,$ and assume that $\rk S_{\cF_0} \ge d-k$. Consider a sequence $\ca= \{\alpha_i\}_{i\in \In} \in \R_{>0}^n\,$ such that $\|f_i\|^2=\alpha_i$ for every $i \in \IN{\n0  }\,$. 

With the fixed data from above, the problem posed in \cite{FicMixPot} is to find a sequence  
$\cF_1= \{f_i\}_{i=\n0  +1}^n\in \cH^{k}$ with $\|f_i\|^2=\alpha_i$, for   
$\n0  +1\leq i\leq n$, such that the the mean square error 
of the resulting completed frame $\cF= (\cF_0\coma \cF_1)= \{f_i\}_{i\in \IN{n}}\in\RS(n \coma d)$, 
namely  $\tr(S_\cF^{-1})$, is minimal among all possible such completions. It is worth pointing out that the mean square error of $\cF= (\cF_0\coma \cF_1)$ depends on $\cF$ through the eigenvalues $\lambda(S_\cF)$ of its frame operator.

\pausa

Note that there are other possible ways to measure robustness of the completed frame $\cF$ as above. 
For example, we can consider optimal (minimizing) completions, with prescribed norms, for the Benedetto-Fickus' potential. In this case we search for a frame $\cF= (\cF_0\coma \cF_1)= \{f_i\}_{i\in \IN{n}}\in\RS(n \coma d)$, with
$\|f_i\|^2=\alpha_i$ for $\n0  +1\leq i\leq n$, and such that its frame potential $\FP(\cF)=\tr(S_\cF ^2)$ is minimal  
among all possible such completions. As before, we point out that the frame potential of the resulting completed frame $\cF= (\cF_0\coma \cF_1)$ depends on $\cF$ through the eigenvalues $\lambda(S_\cF)$ of the frame operator of $\cF$. 

\pausa

Hence, in order to solve both problems above we need to give a step further in the classical frame completion problem (i.e. decide whether $\cF_0$ can be completed to a frame $\cF=(\cF_0,\cF_1)$ with prescribed norms and frame operator $S\in \matpos$) and search for {\it optimal} (e.g. minimizers of the mean square error or Benedetto-Fickus' frame potential) frame completions with prescribed norms.

\pausa

At this point a natural question arises as whether the minimizers corresponding to the mean square error and to the Benedetto-Fickus' potential, or even more general convex potentials, coincide (see \cite{BF,Phys,MR}). 
As we shall see, the solutions of these problems are independent of the particular choice of convex potential considered. Indeed, we show
that under certain hypothesis on the final sequence $\cb=\{\alpha_i\}_{i=\n0+1}^n$ (which includes the uniform case)  
we can explicitly compute the completing sequences $\cF_1 = \{f_i\}_{i=\n0  +1}^n\in \cH^k$ such that the frame operators of the completed sequences $\cF=(\cF_0,\cF_1)$  are minimal with respect to majorization (within the set of frame operators of all completions with norms prescribed by the sequence $\ca$). In order to do this, we begin by fixing some notations.

\begin{fed}\rm 
Let $\cF_0=\{f_i\}_{i\in \IN{\n0  }}\in \cH^{\n0  }$ and 
$\ca= \{\alpha_i\}_{i\in \In} \in \R_{>0}^n\,$ 
such that $d-\rk \, S_{\cF_0} \le n-\n0$ and 
$\|f_i\|^2=\alpha_i$, $i \in \IN{\n0  }\,$. We consider the sets
$$
\cC_\ca(\cF_0)=\big\{\, \{f_i\}_{i\in \IN{n}} \in \RS(n \coma d):
\{f_i\}_{i\in \IN{\n0  }}= \cF_0 \py \|f_i\|^2=\alpha_i \ \mbox{ for }  \ i \ge \n0  +1\,\big\}\ ,
$$ 
\beq
\cS\cC_\ca(\cF_0)=\{S_\cF:\ \cF\in \cC_\ca(\cF_0)\} 
\ . 
\EOEP
\eeq
\end{fed}

\pausa
In what follows we shall need the following solution of the classical frame completion problem.

\begin{pro}[\cite{Illi,MR0}]\label{frame mayo}\rm
Let $B\in \matpos$ with $\lambda(B)\in \R_{+}^d\,^\downarrow$ and let 
$\cb=(\beta_i)_{i\in\IN{k}} \in \R_{>0}^k\,$. Then there exists 
a sequence $\cW=\{g_i\}_{i\in\IN{k}}\in \cH^k$ with frame operator 
$S_\cW= B$ and such that $\|g_i\|^2=\beta_i$ for every $i\in\IN{k}\,$ 
if and only if 
$\cb\prec \lambda(B)$ (completing with zeros if $k\neq d$).
\QED
\end{pro}

\pausa
Since our criteria for optimality of frame completions will be based on majorization,
our analysis of the completed frame $\cF= (\cF_0\coma \cF_1)$ will depend on $\cF$ through $S_\cF$. Hence, the following description of $\cS\cC_\ca(\cF_0)$ plays a central role in our approach.

\begin{pro}\label{con el la y mayo}
Let $\cF_0=\{f_i\}_{i\in \IN{\n0  }}\in \cH^{\n0  }$ and 
$\ca= \{\alpha_i\}_{i\in \In} \in \R_{>0}^n\,$ such that $\|f_i\|^2=\alpha_i$, $i \in \IN{\n0  }\,$. 
Then, we have that 
$$
\cS\cC_\ca(\cF_0) =
\big\{S\in \gld \, :\, 
S\geq S_0  \py (\alpha_i)_{i=\n0  +1}^n \prec \lambda(S-S_0)  
\big\}
\ .
$$ 
In particular, if we let $k=n-\n0  $ then we get the inclusion
\beq\label{inc modelo}
\cS\cC_\ca(\cF_0) \inc\{ S_{\cF_0}+B:\ B\in \matpos \, , \ 
\rk \, B \le k \, , \ \tr(S_{\cF_0}+B)=\sum_{i=1}^n \alpha_i\}\ .
\eeq
\end{pro}
\proof 
Observe that if $\cF = (\cF_0\coma \cF_1) \in \RS(n\coma d)$, 
then $S_{\cF} = S_{\cF_0} + S_{\cF_1}\,$. 
Denote by $ S_0 = S_{\cF_0}$ and 
 $B = S-S_0\,$, for any $S\in \gld$. 
Applying Proposition \ref{frame mayo} to the matrix $B$
(which must be nonnegative if $S\in \cS\cC_\ca(\cF_0)\,$), 
we get the first equality. 

\pausa
The inclusion in Eq. \eqref{inc modelo}
follows using that, if 
$\cF = (\cF_0\coma \cF_1) \in \RS(n\coma d)$, then 
$\rk \, B = \rk S_{\cF_1} \le k = d-(d-k)$. On the other hand, 
recall that $\tr(S_{\cF})=\sum_{i=1}^n \|f_i\|^2$.
\QED

\subsection{Dual frames of a fixed frame with tracial restrictions}\label{duales con rest}

Let $\RSV  \in \Fnd$. Then $\cF$ induces and encoding-decoding scheme as described in 
Eq. \eqref{ec recons}, in terms of the canonical dual $\cF^\#$. But, in case that $\cF$ has nonzero redundancy then we get a family of reconstruction formulas in terms of different frames that play the role of the canonical dual. In what follows we say that $\RSW\in \Fnd$ is a {\bf dual} frame 
for $\cV$ if $ T_\cW^*\, T_\cV = I_\cH\, $ (and hence $ T_\cV^*\, T_\cW = I_\cH\, $), 
or equivalently if the following reconstruction formulas hold:
$$x  =\sum\subim  \api x\coma f_i\cpi  \, g_i =\sum\subim  \api x\coma g_i\cpi  \, f_i   \peso{for every} x\in \H \ .$$
We denote  by  
$$\cD(\cV) \igdef \{\cW\in \Fnd: T_\cW^*\, T_\cV = I_\cH\,\}$$ 
the set of all dual frames for $\cV$. 
Observe that  $\cD(\cV) \neq \vacio$ since $\cV^\#\in \cD(\cV) $.

\pausa
Notice that the fact that $\RSV\in  \Fnd$ implies that $T_\cV^*$ is surjective.
 In this case, a sequence  $\cW\in \cD(\cV)$ \sii  its synthesis 
operator $T_\cW^*$ is a pseudo-inverse of $T_\cV\,$. 
Moreover,  
the synthesis operator  $T_{\cV^\#}^*$ of the canonical dual $\cV^\#$ 
corresponds to the Moore-Penrose pseudo-inverse of $T_\cV\,$. Indeed, notice that 
$T_{\cV}\,T_{\cV^\#}^* = T_{\cV}\, S_{\cV}\inv T_{\cV}^* 
\in L(\C^n)^+$, so that it  
is an orthogonal projection. 
From this point of view, the canonical dual $\cV^\#$ has some optimal properties 
that come from the theory of pseudo-inverses. 
Nevertheless, the canonical dual frame might not be the optimal choice for a dual frame from an applied point of view. 
For example, it is well known that there are classes of structured frames that admit alternate duals that share this structure but for which their canonical duals are not structured (\cite{BLagreg,WES}); in the theory of signal transmission through noisy channels, it is well known that there are alternate duals that perform better than $\cF^\#$ (\cite{LeHanagre,LoHanagre,MRS3}) when we assume that the frame coefficients can be corrupted by the noise in the channel. There are other cases in which $\cF^\#$ may be ill-conditioned or simply too difficult to compute: for example, it is known (see \cite{Han}) that under certain hypothesis we can find Parseval dual frames $\cW\in \cD(\cV)$ (i.e. such that $S_\cW= I_\hil$), which lead to more stable reconstruction formulas for vectors in $\hil$.

In the general case, we can measure the stability of the reconstruction formula induced by a dual frame $\cW\in \cD(\cV)$ in terms of the spread  of the eigenvalues of the frame operator $S_\cW$; this can be seen if we consider, as it is usual in applied situations, the condition number of $S_\cW$ as a measure of stability of linear processes that depend on $S_\cW\,$. There are  
finer measures of the dispersion  which take into account all the eigenvalues of $S_\cW\,$, if one restricts 
to the case of fixed trace. As an example of such a measure we can mention the Benedetto-Fickus' potential. 
Our approach based on majorization - which is the structural measure of the spread of 
eigenvalues for matrices with a fixed trace - allows us to show that minimizers with respect to a large class of convex potentials coincide. 
The main advantages of considering the partition of $\cD(\cF) $ into slices determined by the trace condition $\tr(S_\cW)=t$ are:
\bit
\item There exists a unique vector $\nu(t)$ 
of eigenvalues which is minimal for majorization among the vectors $\la(S_\cW)$, for dual frames $\cW \in \cD(\cF) $ with  
with $\tr\, S_\cW = t$. 
\item Moreover, the vector $\nu(t)$ is also submajorized by the vectors $\la(S_\cW)$ for every 
$\cW \in \cD(\cF) $ with with $\tr\, S_\cW \ge t$. 
\item The map $t \mapsto \nu(t)$ is 
increasing (in each entry) and continuous.
\item Continuous sections $t\mapsto \cW_t \in \cD(\cF)$ such that $\la(\cW_t) 
= \nu(t)$ can be computed.   
\item In addition, the condition number
of $\nu(t)$ decreases when $t$ grows until a critical point (which is easy to compute).
\eit
We point out that both the vector $\nu(t)$ and the duals $\cW_t$ can be computed explicitly in terms of implementable algorithms.
In order to obtain a convenient formulation of the problem we consider the following notions and simple facts.

\begin{fed} \rm
Let $\cV\in \Fnd$. We denote by 
$$\ese \cD(\cV)= \{S_\cW: \cW \in \cD(\cV)\}$$
the set of frame operators of all dual frames for $\cV$.  
\EOE
\end{fed}

\begin{pro} \label{los S 1}
Let $\cV \in \RS(\k \coma d)$. Then 
\beq\label{los S eq 1}
\ese\cD(\cV) = 
\{ S_{\cV^\#} + B:\ B\in \matpos \ \mbox{ \rm and } \ \rk \, B 
\le n-d\} \ .
\eeq
\end{pro}
\proof
Given $\cW \in \Fnd $, then 
$\cW \in \cD(\cV)\iff Z = T_\cW - T_{\cV^\#} \in \lhcn$ 
satisfies $Z^* T_\cV=0$. In this case, by Eq. \eqref{SMP}, we know that 
 $T_{\cV^\#}= T_\cV \, S_\cV\inv \implies Z^* T_{\cV^\#}=0$,  and  
$$
S_\cW = (T_{\cV^\#}+Z)^*\,(T_{\cV^\#}+Z)= 
S_{\cV^\#} + B = S_\cV\inv+B \ , \peso{where} 
B = Z^*Z \in \matpos \ .
$$
Moreover, $S_\cV = T_\cV^*T_\cV \in \gld  \implies \rk \, T_\cV = d$,  
and the equation $T_\cV ^*Z= 0$ implies that 
$$
R(Z ) \inc \ker T_\cV^* = R(T_\cV)\orto \implies \rk \, B 
=   \rk (Z^*Z) = \rk\, Z \le  n-d 
\ .
$$
Since any $ B\in \matpos $ with $\rk \, B \le n-d $ 
can be represented as $B=Z^*Z$ for some  $Z \in L(\cH\coma R(T_\cV)\orto)$, 
we have proved Eq. \eqref{los S eq 1}. 

\QED

\medskip

Fix a system $\RSV \in \Fnd$. Notice that Proposition \ref{los S 1} shows that if $\cW\in \cD(\cV)$ then $S_{\cV^\#}\leq S_\cW$, which is a strong minimality property of the frame operator of the canonical dual $\cF^\#$. 
As we said befire, we are interested in considering alternate duals that are more stable than $\cF^\#$. 
In order to do this, we consider the set $\cD_t(\cV)$ of dual frames $\cW\in \cD(\cV)$ with a further restriction, namely that $\tr(S_\cW)\geq t$ for some $t\geq \tr(S_{\cV^\#})$.
Therefore, the problem we focus in is to find dual frames $\cW_t\in \cD_t(\cV)$ such that 
their frame operators $S_{\cW_t}$ are minimal with respect to submajorization within the set 
\beq\label{defi esecdt} 
\ese\cD_t(\cV) \igdef \{S_\cW: \cW \in \cD_t (\cV)\}\ .
\eeq Notice that as an immediate consequence of Proposition \ref{los S 1} we get the identity
\beq\label{rev defi sdtv}
\ese\cD_t(\cV)=\{ S_{\cV^\#} + B:\ B\in \matpos \ , \ \rk \, B 
\le n-d \ , \ \tr(S_{\cV^\#} + B)\geq t \} \ .
\eeq
As we shall see, these optimal duals $\cW_t$ decrease the condition number and, in some cases are even tight frames. Moreover, because of the relation between submajorization and increasing convex functions, our optimal dual frames $\cW_t\in \cD_t(\cV)$ are also minimizers of a family of convex frame potentials (see Definition \ref{pot generales} below) that include the Benedetto-Fickus' frame potential.

\subsection{A unified matrix model for the frame problems and submajorization}\label{unified}

In this section we introduce and develop some aspects of a set 
$U_t(S_0\coma m)\inc\matpos$ that will play an essential role 
in our approach to the frame problems described above (see Remark \ref{estan en U}). Our main results related with $U_t(S_0\coma m)$
are Theorem \ref{el St general} and Proposition \ref{pdefi del nulam}. 
In order to avoid some technicalities, we postpone their proofs to the Appendix (Section \ref{appendixiti}).

\begin{fed}\label{conjs ese}\rm 
Let $S_0\in \matpos$ with $\lambda(S_0)=\lambda\in 
\R_{+}^d\,^\downarrow\,$,    
$ t_0= \tr \,S_0\,$,  and $t\ge t_0\,$. For any integer 
$ m< d$ we consider the following subset of $\matpos$ :
\beq\label{defi usubt}
U_t(S_0\coma m)=\{S_0+B:\   B\in \matpos \, , \ \rk \, B 
\le d-m \ , \ \tr(S_0+B)\geq t\ \} \ .
\eeq
Observe that if $m\le 0$ then $U_t(S_0\coma m) = 
\{S\in \matpos \, : \,  S\ge S_0\, , \ \tr(S)\geq t\,\}$. \EOE
\end{fed}

\begin{rem}\label{estan en U}
As a consequence of Eq. \eqref{inc modelo} and Eq. \eqref{rev defi sdtv} we see that the 
two main problems are intimately related with the structure of the set $U_t(S_0\coma m)$ 
for suitable choices of the parameters $S_0\in\matpos$, $m<d$ and $t\geq \tr \, S_0\,$ : 
\ben 
\item Note that Eq. \eqref{inc modelo} shows that $\cS\cC_\ca(\cF_0) \inc U_t(S_{\cF_0}\coma m)$, 
where $t=\tr \ca$ and $m=d-n+\n0  \,$. 
\item Similarly, 
 Eq. \eqref{rev defi sdtv} shows that identity $\ese\cD_t(\cV)=U_t(S_{\cF^\#}\coma m)$ where $m=2d-n$.\EOE
 \een
\end{rem}

\begin{rem}\label{rem imag espec usubt}
Given $\la = \la(S_0)\in \R_{+}^d\,^\downarrow $ 
and $m<d$ we look for a $\prec_w$-minimizer  on the set 
\beq\label{defi imag espec usubt}
\Lambda (U_t(S_0\coma m)\,) \igdef \{ \la(S): S\in U_t(S_0\coma m)\}
\inc \R_{+}^d\,^\downarrow \ .
\eeq
Heuristic computations suggest that 
in 	some cases such a minimizer should have the form 
$$
\nu = (\la_1\coma \dots\coma \la _r\coma c\coma \dots c) \in \R_{>0}^d\,^\downarrow 
\quad \mbox{with \ \ $\tr \, \nu = t$ \ \ \ for some \quad $r \in \IN{d-1}$ \ and \  $c\in \R_{>0}$}\ .
$$
Observe that if $\nu \in \Lambda (U_t(S_0\coma m)\,)$ 
 then $ \la \leqp \nu = \nu^\downarrow\,$. 
Hence we need that 
$$
c= \frac{ t-\sum_{j=1}^r\lambda_j }{d-r} \py 
\la_{r+1}\le c \le \la_{r} \ . 
$$
These restrictions on the numbers $r$ and $c$ suggest the following definitions: 
\EOE
\end{rem}

\begin{fed}\label{defi irregularidad} \rm
Let $\lambda\in \R_{+}^d\,^\downarrow$ and $t\in \R$ such that 
$ \tr \lambda \le t < d \, \la_1\,$. 
Consider the set 
$$
A_\la(t) \igdef \big\{\, r\in\IN{d-1} : 
\ p_\la(r\coma t) \igdef \ \frac{\barr{rl} & t-\sum_{j=1}^r\lambda_j \earr}{d-r}\ 
\geq \lambda_{r+1} \ \big\} \ .
$$ 
Observe that $ t\ge \tr \la \implies t-\suml_{j=1}^{d-1}\lambda_j\geq \lambda_{d}\,$, so that 
$d-1\in A_\la(t)\neq \vacio\,$.
The $t$-irregularity of the ordered vector $\lambda$, denoted $r_\la(t)$, is defined by 
\beq\label{el rt}
r_\la(t)  \igdef \min \, A_\la(t) = \min \{r\in\IN{d-1} : p_\la(r\coma t) 
\geq \lambda_{r+1}\}\ . 
\eeq
If $t\ge d\,\la_1\,$, we set $r_\la(t) \igdef 0$ and $p_\la(0\coma t) = t/d\ $. \EOE
\end{fed}

\pausa
For example, if $t_0 = 
\tr \,\la\,$ then for every $r \in \IN{d-1}$ we have that 
$$
p_\la(r\coma t_0) = \frac{t_0-\sum_{j=1}^r\lambda_j}{d-r} = 
\frac{\sum_{j=r+1}^d\lambda_j}{d-r} \ge \la_{r+1} 
\iff \la_{r+1} = \la_d \ .
$$
Therefore in this case 
\bit 
\item If $\la = c\,\uno_d$ for some $c\in \R_{>0}\,$,  then $r_\la(t_0) = 0$. 
\item If $\la_1>\la_d\,$, then 
\beq\label{t = tr} 
r_\la(t_0)+1  = \min \{i \in \IN{d} : \la_i = \la_d\} \py 
r_\la(t_0) = \max\{r\in  \IN{d-1} : \la_r>\la_d\} \ .
\eeq 
\eit

\begin{fed}\label{defi r c}\rm
Let $\lambda\in \R_{+}^d\,^\downarrow$ and $ t_0 = \tr \la$. 
We define the functions \begin{equation}\label{defi rs}
r_\la :[t_0\coma +\infty)\rightarrow \{0,\ldots,d-1\}
\peso{given by} r_\la(s)\ \stackrel{\eqref{el rt}}{=} \text{ \  the $s$-irregularity of }\lambda 
\end{equation}
\begin{equation}\label{defi cs}
c_\la:[t_0\coma +\infty)\rightarrow \R_{\geq 0}\peso{given by} 
c_\la(s) = p_\la(r_\la(s)\coma s)= \frac{s-\sum_{i=1}^{\,r_\la(s)}\lambda_i}{d-r_\la(s)}  \ \ ,
\end{equation}
for every $s\in [t_0\coma +\infty)$, where we set 
$\suml_{i=1}^{0}\lambda_i =0$. 
\EOE
\end{fed}

\pausa
Fix $\lambda\in \R_{+}^d\,^\downarrow$. As we shall show in Lemma \ref{r y c}, 
the vector 
$\nu = (\la_1\coma \dots\coma \la _{r_\la(t)} \coma c_\la(t) \, \uno_{d-r_\la(t)}) 
\in \R_{>0}^d\,^\downarrow $ for every $t\ge t_0\,$,  
and the map $c_\la$ is piece-wise linear, strictly increasing and continuous. This last claim allows us to introduce the following parameter: given $m\in \IN{d-1}\,$ we denote by 
\beq\label{el s* def} 
s^*=s^*(\la\coma m) \igdef c_\la\inv (\la_m)= \suml_{i=1}^{m} \la_i + (d-m) \, \la_{m} 
\eeq 
that is, the unique $s\in [t_0\coma+\infty)$ such that $c_\la(s) = \la_m\,$. These facts and other 
results of Section \ref{appendixiti} give consistency to the following definitions:

\begin{fed}\label{el d y r prima} \rm
Let $\lambda\in \R_{+}^d\,^\downarrow$, 
$ t_0= \tr \lambda $. Take an integer  $m<d$.  
If $m> 0$ 
and $t\in [t_0 \coma +\infty)$ let
$$
c_{\la\coma m}(t) \igdef \begin{cases} \ c_\la(t) & \mbox{if} \ \ t\le s^*  \\&\\
\la_m+\frac{t-s^*}{d-m} & \mbox{if} \ \ t> s^* 
\end{cases} \quad \quad \py
$$
$$
r_{\la\coma m}(t) \igdef 
\min\limits \{r\in\IN{d-1} \cup \trivial : c_{\la\coma m}(t)  
\geq \lambda_{r+1}\}
  \ \ .
$$
If $m\le 0$ and $t\in [t_0 \coma +\infty)$ we define $c_{\la\coma m}(t) = c_{\la}(t) $ 
and $r_{\la\coma m}(t) = r_{\la}(t) $. \EOE
\end{fed}

\pausa
Note that, by Eq. \eqref{circular} of Lemma \ref{r y c}, 
$r_{\la\coma m}(t)= r_{\la}(t)$ for every $t\le s^*\,$. 
The following results will be used throughout Section \ref{la solucion frames}; see the Appendix (Section \ref{appendixiti}) for their proofs.
 
\begin{teo}\label{el St general}
Let $S_0\in \matpos$ with  $\la = \la(S_0)$ and $m<d $ be an integer.  
For $t\ge \tr \, S_0\,$, let us denote by 
 $r' =\max\{ r_{\la\coma m}(t), m\}$ and $c = c_{\la\coma m}(t)$. 
 Then, there exists $\nu  \in \Lambda(U_t(S_0\coma m)\,)$  such that 
\ben
\item The vector $\nu$ is $\prec_w$-minimal in $\Lambda(U_t(S_0\coma m)\,)$,  i.e. $\nu\prec_w\mu$ for every $\mu\in\Lambda(U_t(S_0\coma m)\,)$. 
\item 
 For every matrix 
$S  \in U_t(S_0\coma m)$ the following conditions are equivalent:
\rm
\ben
\item \it  $\lambda(S) = \nu$ (i.e. $S$ is $\prec_w$-minimal in $U_t(S_0\coma m)$). 
\rm
\item 
\it 
There exists $\{v_i\}_{i\in \IN{d}}\,$, an ONB  
of eigenvectors for  $S_0 \coma \la$  such that   
\beq\label{Formula de B}
B= S-S_0 =  \suml_{i=1}^{d-r'} (c -\la _{r'+i}) \, 
v_{r'+i} \otimes v_{r'+i}  \ .
\eeq
\een
\item \it 
If we further assume any of the following conditions: 
\bit
\item  $m\le 0$,
\item $m\ge1$ and  $\la_m >\la_{m+1} \,$, or 
\item  $m\ge1$ and  $\la_m =\la_{m+1} \,$ but 
$t\le s^*(\la\coma m) $ (see Eq. \ref{el s* def}), 
\eit 
then $B$ and $S$ are unique. 
Moreover, in these cases Eq. \eqref{Formula de B} holds for any  ONB 
of eigenvectors of $S_0$ as above. \QED
\een
\end{teo}

\begin{rem}\label{clm y cl} 
Suppose that $m\geq 1$. In this case the map $c_{\la\coma m}(\cdot)$ is continuous and 
strictly increasing.
Indeed, 
by Lemma \ref{r y c} we know that 
$s^* = \sum_{i=1}^{m} \la_i + (d-m) \, \la_{m} \,$. 
Hence 
\beq\label{also}
c_{\la\coma m} (t)= \la_m+\frac{t-s^*}{d-m} = 
\frac{t-\sum_{j=1}^m\lambda_j}{d-m}  \peso{for every} t>s^* \ . 
\eeq
The fact that  the map $c_\la$ is continuous and 
strictly increasing will be also proved in Lemma \ref{r y c}. 
Let us abbreviate by $r = r_{\la\coma m}(t)\,$ for any fixed $t>s^*$.
Then, if $r>0$ we have that 
\beq\label{c y r}
r <m \py  \la_{r} \ge c_{\la\coma m}(t)=\la_m+\frac{t-s^*}{d-m} \ge \la_{r+1}
\ .
\eeq
\pausa
Finally, notice that the previous remarks  allow to define
\begin{equation}\label{defi s**}
s^{**} = c_{\la\coma m}\inv(\la_1)
\stackrel{\eqref{also}}{=} (d-m)\, \la_1 + \suml_{j=1}^m \la_j \ge s^*  
\quad \mbox{(with equality $\iff \la_1 = \la_m$) . }
\end{equation}
Then $c_{\la\coma m}(t) \ge \la_1$ and $r= r_{\la\coma m}(t)= 0$ for every $t>s^{**}$ (by Definition \ref{el d y r prima}). 
These remarks are necessary  to characterize the vector $\nu$ of Theorem \ref{el St general}: 
\EOE
\end{rem}
 
\begin{pro}\label{pdefi del nulam} 
Let $S_0\in \matpos$ with  $\la = \la(S_0)$, 
$ t_0= \tr(S_0)$ and $m\in \mathbb{Z}$ such that $m<d$.  
Fix $t \in [t_0 \coma+\infty)$ and denote by 
$r = r_{\la\coma m}(t)\,$. Then, the minimal vector
$\nu=\nu(\la\coma m\coma t)\in \R_{+}^d\,^\downarrow \,$ of Theorem \ref{el St general} 
has $\tr \, \nu = t$ and it  is
given by the following rule: 
\bit
\item 
If $m\le 0$ then 
$\nu=\big( \la_1\coma \dots \coma \la_{r} \coma c_{\la\coma m}(t) \, \uno_{d-r}\big)
=\big( \la_1\coma \dots \coma \la_{r} \coma c_{\la}(t) \, \uno_{d-r}\big)$.   
\eit
If $m\ge 1$ we have that
\bit
\item 
$\nu = \big( \la_1\coma \dots \coma \la_{r} \coma 
c_{\la\coma m}(t) \, \uno_{d-r}\big)$ for $t\le s^*$ 
(so that  $r\ge m$ and $c_{\la\coma m}(t)\le \la_m$). 
\item 
$\nu =  \Big( \la_1\coma \dots \coma \la_{r} 
\coma c_{\la\coma m}(t) \, \uno_{d-m}\coma \la_{r+1}\coma 
\dots \coma \la_m\Big) $
for $t\in (s^*\coma s^{**})$, and  
\item 
$\nu  = \big( c_{\la\coma m}(t) \, \uno_{d-m}\coma 
\la_{1}\coma \dots \coma \la_m\big)$ 
for $t\ge s^{**}\,$.
\eit
If $\la_1=\la_m\,$, the second case above disappears.\QED

\end{pro}

\section{Solutions of the main problems}\label{la solucion frames}

In this section we present the solutions of the problems in frame theory described in Section \ref{description of main prob}. Our strategy is to apply Theorem \ref{el St general} and Proposition \ref{pdefi del nulam} to the 
matrix-theoretic reformulations of these problems obtained in Sections \ref{Fick} and \ref{duales con rest}. 
We point out that our arguments are not only constructive but also algorithmically implementable. 
This last fact together with recent progress in algorithmic constructions of solutions to the classical frame design problem allow us to effectively compute the optimal frames from Theorems \ref{teo F1} and \ref{El S bis} below (for optimal completions see Remarks \ref{rem algo 1}, \ref{rem algo 2} and Section \ref{exas}; for optimal duals see Remark \ref{rem algo 3} and Example \ref{exa dual}).

\subsection{Optimal completions with prescribed norms}\label{opti completions}

Next we show how our previous results and techniques allow us to partially 
solve the frame completion problem described in Section \ref{Fick} (which includes the problem posed in \cite{FicMixPot}).
We begin by extracting the relevant data for the problem: 
\begin{num}\label{data}
In what follows, we fix the following data: 
A space $\cH \cong \C^d$. 
\ben
\item[D1.] A sequence of vectors $\cF_0=\{f_i\}_{i\in \IN{\n0  }}\in \cH^{\n0  }$.  
\item[D2.] An integer $n>\n0  \,$. We denote by $k = n-\n0  \,$.
We assume that $\rk S_{\cF_0} \ge d-k$. 
\item[D3.] A sequence $\ca= \{\alpha_i\}_{i\in \In} \in \R_{>0}^n\,$
such that $\|f_i\|^2=\alpha_i$ for every $i \in \IN{\n0  }\,$. 
\item[D4.] We shall denote by $t = \tr \, \ca$ 
and by $\cb = \{\alpha_i\}_{i=\n0  +1}^n \in \R_{>0}^k\,$. 
\item[D5.] 
The vector $\lambda=\lambda(S_{\cF_0})\in \R_{+}^d\,^\downarrow$.
\item[D6.] The integer  $m=d-k = (d+\n0  )-n$. Observe that  $d-m=k = n-\n0  \,$.
\EOE
\een
\end{num}

\pausa
In order to apply the results of Section \ref{unified} to this problem, we need to recall and restate some
objects and notations: 

\begin{fed}\label{el nu tilde}\rm
Fix the data $\cF_0=\{f_i\}_{i\in \IN{\n0  }}$ and $\ca= \{\alpha_i\}_{i\in \In} $
as in \ref{data}. Recall that $t = \tr \ca$, $\la= \la(S_{\cF_0})$  and $m=d-k$. We rename 
some notions of previous sections: 
\ben 
\item   The vector $\hn= \nu
 \in\R_{\geq 0}^d\,$ of Theorem \ref{el St general} (see also Proposition \ref{pdefi del nulam}).  
\item The number $c = \hc \igdef c_{\la\coma m}(t)$ 
(see Definition \ref{el d y r prima}). 
\item The integer  $r = \hr \igdef \max \{r_{\la\coma m}(t)\coma m\}$ 
(see Definition \ref{el d y r prima}). 
Note that  $d-r\le k$. 
\item Now we consider the vector $$\mu = \hm \igdef 
\Big(\,\hc -\la _{r+j}\Big)_{j \in \IN{d-r}} \in (\R_{\geq 0}^{d-r})^\uparrow\, .$$ Observe that 
$\tr \, \mu = \tr \, \nu - \tr\, \la = 
t - \suml_{i\in \IN{\n0  }} \|f_i\|^2
= \tr\, \ca  - \suml_{i\in \IN{\n0  }} \al_i = \tr\, \cb$. 
\EOE
\een
\end{fed}
                           
\pausa                           
Throughout the rest of this section 
we shall 
denote by $S_0 = S_{\cF_0}$ the frame operator of $\cF_0\,$. 
Recall the following notations of Section \ref{Fick}: 
$$
\cC_\ca(\cF_0)=\big\{\, \{f_i\}_{i\in \IN{n}} \in \RS(n \coma d):
\{f_i\}_{i\in \IN{\n0  }}= \cF_0 \py \|f_i\|^2=\alpha_i \ \mbox{ for }  \ i \ge \n0  +1\,\big\}\ ,
$$ 
$$
\cS\cC_\ca(\cF_0)=\{S_\cF:\ \cF\in \cC_\ca(\cF_0)\} 
\peso{and} \Lambda_\ca(\cF_0) \igdef\{\lambda(S): \ S\in \cS\cC_\ca(\cF_0)\}
\ . 
$$

\begin{teo} \label{teo F1}
Fix the data of \ref{data} and \ref{el nu tilde}. 
If we assume that $\cb\prec \hm$ then 
 \begin{enumerate}
 \item The vector $\nu = \hn\in\Lambda_\ca (\cF_0 )$. 
\item We have that $\nu \prec \beta $ for every other $\beta \in \Lambda_\ca (\cF_0 )$.
\item 
Let $r= \hr$.  Given
$\cF_1 = \{f_i\}_{i=\n0  +1}^n\in \cH^{k}$ 
such that $\cF = (\cF_0 \coma \cF_1) \in \cC_\ca(\cF_0) $,  the following conditions are equivalent:
\rm
\ben
\item \it  $\lambda(S_\cF) = \nu$ (i.e. $S_\cF$ is $\prec$-minimal in $\cS\cC_\ca(\cF_0)$). 
\rm
\item 
\it 
There exists $\{h_i\}_{i\in \IN{d}}\,$, an ONB  
of eigenvectors for  $S_0 \coma \la(S_0)$
  such that   
\beq\label{El B de F}
S_{\cF_1}=  B = \suml_{i=1}^{d-r} \mu_i \, h_{r+i} \otimes h_{r+i} \ . 
\eeq
\een
Since, by the hypothesis, $\cb\prec \mu = \hm$ then such an $\cF_1$ exists. 
\item Moreover, if 
any of the conditions in item 3 of Theorem \ref{el St general} holds,  
then 
\ben
\item 
Any ONB of eigenvectors for $S_0\coma \la $ produces the same 
operator $B$ via \eqref{El B de F}. 
\item Any $\cF = (\cF_0 \coma \cF_1) \in \cC_\ca(\cF_0) $
satisfies that $\la (S_\cF) = \hn  \iff S_{\cF_1} = B$. 
\een
\end{enumerate}
\end{teo}
\proof
Since the elements of $\cC_\ca(\cF_0) $ must be frames, we have first 
to show that $\hn>0$. 
By the description of 
$\nu = \hn
\in \R_{+}^d\,^\downarrow \,$ given in Proposition \ref{pdefi del nulam}, 
there are two possibilities: 
In one case  $\nu_d = \la_m$ which is positive because we know from the data 
given in \ref{data} that $\rk \, S_0 \ge m$. 
Otherwise $t\le s^*$ so that 
$\nu_d = \hc = c_{\la}(t)$ by Proposition \ref{pdefi del nulam} and Definition 
 \ref{el d y r prima}. But  $c_{\la}(t)>0$
because  $\cb>0 \implies t > \tr S_0\,$ 
(see Lemma \ref{clm y cl} and Definition \ref {defi r c}). 

\pausa
By Proposition \ref {frame mayo}, we know that 
the hypothesis $\cb\prec \mu = \hm$ assures that there exists a 
sequence $\cF_1 = \{f_i\}_{i=\n0  +1}^n\in \cH^{k}$  
such that $\cF = (\cF_0 \coma \cF_1) \in \cC_\ca(\cF_0) $ and 
$S_{\cF_1} = B$. Then 
$$
\la(S_\cF) = \la(S_{\cF_0} + S_{\cF_1}) = \la(S_{\cF_0} + B) = \hn \in \Lambda_\ca (\cF_0 ) \ ,
$$ 
by Theorem \ref{el St general}. 
Observe that  $\Lambda_\ca (\cF_0 ) \inc 
\Lambda (U_t(S_{\cF_0}\coma m )\,)$, 
by 
Remark \ref{estan en U}. 
Hence the majorization of item 2, the equivalence of item 3 
 and the uniqueness results of item 4 follow from 
Theorem \ref{el St general}. Note that all the vectors of $\Lambda_\ca (\cF_0 )$ 
have the same trace. So we have $\prec$ instead of $\prec_w\,$. 
\QED

\begin{teo}\label{teo F2}
Fix the data of \ref{data} and \ref{el nu tilde}.
If we assume that $\cb\prec \hm$ then 
\ben
\item 
Any  $\cV\in \cC_\ca(\cV_0)$ such that $\lambda(S_{\cV})=\hn$ satisfies that 
\begin{equation*}\label{eq hay min2}
\sum_{i\in \IN{d}}f(\hn_i)=P_f(\cV)\leq P_f(\cW) \  \ 
\text{ for every } \cW\in \cC_\ca(\cV_0)\,, 
\end{equation*} 
and every (not necessarily increasing) convex function $f: [0,\infty)\rightarrow [0,\infty)$.
\item 
If  $f$ is strictly convex then, for every global minimizer 
$\cV'$ of $P_f(\cdot)$ on $\cC_\ca(\cV_0)$
we get that $\lambda(S_{\cV'})=\hn$. 
\een
In particular the previous items holds for the Benedetto-Fickus' potential and the mean square error.
\end{teo}
\proof It follows from Theorem \ref{el St general} and the majorization 
facts described in Section \ref{subsec 2.2}.
\QED

\pausa
Fix the data $\cF_0$ and $\cb$ of \ref{data} and \ref{el nu tilde}.
We shall say that ``the completion problem is {\bf feasible}" 
if the condition $\cb\prec \hm$ of Theorem \ref{teo F1} is satisfied.

\begin{rem}\label{rem algo 1} \rm
The data $\hn$, $\hr$, $\hc$ and  $\hm$ are essential for 
Theorem \ref{teo F1}, both for checking the feasibility 
hypothesis $\cb\prec \hm$ 
and for the construction of the matrix $B$ of \eqref{El B de F}, 
which is the frame operator of the optimal extensions of $\cF_0\,$.
Notice that the vector $\hm$ measures
how restrictive is the feasibility condition. 
Fortunately, this condition can be easily computed according to the following algorithm:

\ben
\item The numbers $t = \tr \, \ca$ and $m = d-k$ are included in the data \ref{data}.
\item The main point is to compute the irregularity 
$r = \hr = \max \{r_{\la\coma m}(t)\coma m\}$.  If $m\leq 0$ then 
\eqref{el rt} allows us to compute $r_\la(t)$.  
If $m\ge1$, the number $s^* = s^*(\la\coma m)$ of Eq. \eqref{el s* def}  
allows us  to compute 
$r\,$: If $t> s^*$ then $r=m$ by Eq. \eqref{c y r}, and if $t\le s^*$ 
then $r = r_\la(t)$ by the remark which follows Definition \ref{el d y r prima}.  
 
\item Once $r$ is obtained, we can see that 
the wideness of the allowed weights $\cb$ depends 
on the dispersion of the eigenvalues 
$(\la_{r+1} \coma \dots \coma \la_d)$ of $S_{\cF_0}\,$. 

\item Indeed, the number  
 $t_1\igdef  \tr \, \cb = t -\tr \, S_{\cF_0}$ is known data. 
Also $\tr \hm = t_1\,$. Hence 
$\hc$ and  $\hm$ can be directly computed: 
Let $s = \sum_{i=r+1}^d \, \la_i\,$. Then 
$$
t_1 = \tr \mu = (d-r) \, \hc - s \implies 
\hc = \frac{t_1 + s}{d-r} = 
\frac{\tr \, \cb + \sum_{i=r+1}^d \, \la_i}{d-r} \ \ .
$$
And we have the vector $\mu = \hm = 
\Big(\,\hc -\la _{r+j}\Big)_{j \in \IN{d-r}} \in (\R_{\geq 0}^{d-r})^\uparrow \,$. 
Then 
$$
\cb\prec \mu \iff \sum_{i =1}^p \cb_i^\downarrow + 
\la _{d-i+1} \le 
\frac {p}{d-r} \  \big(\,\tr \, \cb + \sum_{i=r+1}^d \, \la_i\,\big)
\peso {for}
1\le p < d-r \ ,
$$
since the last inequalities 
$s+ \sum_{i =1}^p \cb_i^\downarrow \le s+\tr \, \cb$ (for $d-r \le p \le k$) clearly hold.  
\een
It is interesting to note that the closer $\cF_0$ is to be tight 
(at least in the last $r$ entries of $\la$), 
the more restrictive Theorem \ref{teo F1} becomes; but in this case 
$\cF_0$ and $\cF_0^\#$ are already ``good".

\pausa
On the other hand, if $\cF_0\,$ is far from being tight then the sequence  
$(\la_{r+1} \coma \dots \coma \la_d)$ has more dispersion
and the feasibility condition $\cb\prec \hm$ becomes less restrictive. 
It is worth mentioning that in the uniform 
case $\cb = b\, \uno_k\,$ is always feasible and 
Theorem \ref{teo F1} can be applied. 

\pausa
Observe that as the number $k$ of vectors increases
(or as the weights $\alpha_i$ increase) the trace $t$ 
grows and the numbers $r$ and $m$ become smaller, taking into account more entries 
$\la_i$ of  $\la(\cF_0)\,$. This fact offers a criterion 
for choosing a convenient
data  $k$  and $\cb$ for the completing process. We remark that 
the vector $\mu$ (and therefore the feasibility) only depends
on $\la$, $k$  and {\bf the trace} of $\cb$, so the feasibility 
can also be obtained by changing $\cb$ maintaining its length (size) and its trace.   

\pausa
The above algorithm (which tests the feasibility  of our method 
for fixed data $\cF_0$ and $\cb$) can be easily implemented  
in M{\scriptsize ATLAB} with low complexity (see Section \ref{exas}). 
\EOE
\end{rem}

\begin{rem}[Construction of optimal completions for the mean square error]\label{rem algo 2}  
 Consider the data in \ref{data}. Apply the algorithm described in Remark \ref{rem algo 1} and assume that $\cb\prec \hm$.  Then construct $B$ as in Eq. \eqref{El B de F}. In order to obtain an optimal completion of $\cV_0$ with prescribed norms we have to construct a sequence $\cV_1\in \cH^k$ with frame operator $B$ and norms given by the sequence $\cb$
 (which is minimal for the mean square error  by Theorem \ref{teo F2}).
But once we know $B$ and the weights $\cb$ we can apply the results in \cite{CFMP} in order to concretely construct the sequence $\cF_1\,$. 
In fact,  in \cite{CFMP} they give a M{\scriptsize ATLAB} implementation which 
works fast, with low complexity. 

\pausa
We also implemented a M{\scriptsize ATLAB} program which compute the matrix $B$
as in Eq. \eqref{El B de F}. This process is direct, but it is more complex because it depends on finding a ONB of eigenvectors
for the matrix $S_{\cF_0}\,$. In Section \ref{exas} we shall present 
several examples which use these programs for computing explicit
solutions. 
\EOE
\end{rem}

\pausa
Consider the data in \ref{data} and \ref{el nu tilde}. Let $f:[0,\infty)\rightarrow [0,\infty)$ be an strictly convex function. By Remark \ref{estan en U}, Theorem \ref{el St general} and the remarks in Section \ref{subsec 2.2},  in general we have that  \beq\label{unamasy} \sum_{i\in \IN{d}}f(\hn_i) \leq  P_f(\cF)\peso { for every } \cF\in \cC_\ca(\cV_0)\ .\eeq Notice that although the left-hand side of Eq.\eqref{unamasy} can be effectively computed, the inequality might not be sharp. Indeed, Eq.\eqref{unamasy} is sharp if and only if the completion problem is feasible and, in this case, the lower bound is attained if and only if $\la(S_{\cF})=\hn$.  Nevertheless, Eq.\eqref{unamasy} provides a general lower bound that can be of interest for optimization problems in $\cC_\ca(\cV_0)$.

\subsection{Examples of optimal completions with prescribed norms}\label{exas}

 In this section we show several examples obtained by 
implementing the algorithms described in Remarks \ref{rem algo 1} and  \ref{rem algo 2} in M{\scriptsize ATLAB}, for different choices of $\cF_0=\{f_i\}_{i\in \IN{\n0  }}$ and $\ca= \{\alpha_i\}_{i\in \In} $
(as in \ref{data}).
Indeed, we have implemented the computation of $r \coma c \coma \mu$ and $\nu$  
by a fast algorithm using $\cb = \{\alpha_i\}_{i=\n0  +1}^n \in \R_{>0}^k\,$ and the vector $\la=\la(S_{\cF_0})$ as data. Then, after 
computing the eigenvectors of  $S_{\cF_0}$ with the function 'eig' 
in M{\scriptsize ATLAB} we computed the matrix $B$, and 
we apply the one-sided Bendel-Mickey algorithm 
(see \cite{DHST} for details) to construct the vectors of
$\cF_1$ satisfying the desired properties. The corresponding M-files that compute all the previous objects are freely distributed by the authors. 

\begin{exa}\label{Ej1} 
Consider the frame $\cF_0\in \RS(7\coma 5)$ 
 whose analysis operator is 
\[
T_{\cF_0}^* = {\scriptsize  
\left[
 \begin{array}{rrrrrrr}
    0.9202&   -0.7476&   -0.4674&    0.9164&    0.1621&    0.3172&   -0.5815\\
    0.4556&    0.0164&    0.0636&    1.0372&   -1.6172&    0.3688&    0.2559\\
   -0.0885&   -0.3495&   -0.9103&    0.3672&   -0.6706&   -0.9252&    0.6281\\
    0.1380&   -0.4672&   -0.6228&   -0.1660&    0.9419&    1.0760&    1.1687\\
    0.7082&    0.2412&   -0.1579&   -1.8922&   -0.4026&    0.1040&    1.6648\\
 \end{array}
 \right]}
 \ .\]
The spectrum of it frame operator is  
$\la=\la(S_{\cF_0})=(9, 5, 4, 2, 1) $ and 
 $t_0 = \tr\,S_{\cF_0}= 21$. 
As in \ref{data}, fix the data 
$k=2$ and $\cb = \{\alpha_i\}_{i=8}^{9}=(3\coma 2.5) \in \R_{>0}^2\,$,  
so that $m = d-k = 3$. We compute: 
\ben
\item The number $r_{\la\coma m}(26.5) = 2$ and the vector $\mu=(2.25\coma 3.25)$. 
Notice that, in this case, $\cb =(3\coma 2.5) \prec (2.25\coma 3.25)=\mu$. Therefore the
completion problem is feasible. 
\item The optimal spectrum is 
$\nu_{\la\coma m}(26.5) = (9\coma 5\coma 4.25 \coma 4.25 \coma 4)$. 
\item 
An optimal completion $\cF_1$ of $\cF_0\,$, with squared norms given by $\cb$ is given by:
\[
T^*_{\cF_1} = {\scriptsize 
\left[
 \begin{array}{rr}
   -0.6120&   -1.1534\\
    0.9087&    0.1097\\
   -1.0680&    0.7154\\
    0.3735&    0.7676\\
   -0.1404&   -0.7462\\
 \end{array}
 \right]}
 \ .\]
\item If we take $\cb = (3.5 \coma 2)$ then the number 
$t = t_0 + \tr\, \cb$ (and so also  $r$ and $\mu$) are the same 
as before but the problem is not feasible, 
because in this case $\cb \not \prec \mu$. 
\een
\end{exa}

\begin{exa}
We want to complete frame $\cF_0$ 
of Example \ref{Ej1} with 4 vectors in $\R^4$, whose norms are given by $\cb=\{\alpha_i\}_{i=8}^{11}=(1\coma 1 \coma \frac{1}{2} \coma \frac{1}{4}) \in \R_{>0}^4$. We can compute that  
\ben
\item $m=d-k=1$ and $r=r_{\la\coma m}(23.75)=3$. 
\item The vector $\mu=(0.875\coma 1.875)$, 
so that $\cb\prec \mu$ and the problem is feasible. 
\item The optimal spectrum is 
$\nu=(9\coma 5\coma 4 \coma 2.875 \coma 2.875)$. 
\item 
An example of optimal completion $\cF_1$ of $\cF_0\,$, 
with squared norms given by $\cb$ is given by:
$$
T^*_{\cF_1}= {\scriptsize 
\left[
 \begin{array}{rrrr}
   -0.7086&   -0.3232&   -0.5011&   -0.3543\\
    0.1730&    0.5746&    0.1224&    0.0865\\
    0.2597&   -0.7252&    0.1836&    0.1299\\
    0.4674&    0.1935&    0.3305&    0.2337\\
   -0.4267&   -0.0457&   -0.3017&  -0.2133\\
 \end{array}
 \right]}
 \ . 
$$
\item If we take $\cb=(2\coma \frac{1}{4} \coma \frac{1}{4} \coma \frac{1}{4}) \in \R_{>0}^4$ 
the problem becomes not feasible.  \EOE
\een
\end{exa}

\begin{exa}
Suppose now that $\cH = \C^5$ and that 
our original set of vectors $\cF_0=\{f_i\}_{i\in \IN{6}}\in \cH ^6$ 
is such that the spectrum of $S_{\cF_0}$ is given by
$\la=(7\coma4\coma4\coma 3\coma 1)$.  Thus $t_0=\tr \, S_{\cF_0}=19$. 
Let $\cb=(2\coma 2\coma 1)$. Then, $k=3$, $m=d-k=2$ and $t=24$.

\pausa
With these initial data, we obtain the values $r_{\la\coma m}(24)=1$ and  $c_{\la\coma m}(24)=4.33$.
The spectrum of the completion $B$ is $\mu=(0.33\coma 1.33\coma 3.33)$ (notice that $\cb \prec \mu$) and the optimal spectrum is
$\nu_{\la\coma m}(24) = (7\coma 4.33\coma 4.33 \coma 4.33 \coma 4)$.
However the frame operator $B$ of the optimal completion is not unique, since $\la_m=\la_{m+1}$ and  $t=t_0+\tr \, \cb=24>23=s^*$ 
(see  Theorem \ref{el St general}).
\EOE 
\end{exa}

\subsection{Minimizing potentials in $\cD_t(\cV)$}\label{min duals}

In this section we show how our previous results and techniques allow us to solve the problem of computing optimal duals in $\cD_t(\cF)= \{\cW\in \Fnd : T_\cW^*\,T_\cF = I $ and $ \tr \,S_\cW \ge t\}$ 
for a given frame $\cF$, described in Section \ref{duales con rest}.
In order to state our main results we introduce the set $\Lambda_t(\cD(\cV)\,)$, called the spectral picture of the set $\ese \cD(\cV)$ (see Eq. \eqref{defi esecdt}), given by 
$$
\Lambda_t(\cD(\cV)\,)= \{\la(S_\cW)\, :\, \cW\in \cD_t(\cV)\} \ .
$$

\begin{rem}\label{cacho prueba}
Recall from Remark \ref{estan en U}  that if  $\cF \in \Fnd$ with $\la = \la(S_\cV\inv)\,$, $m  = 2\, d-n$ and 
 $t\ge \tr \la$, then $\ese\cD_t(\cV)=U_t(S_{\cF^\#}\coma m)$. Hence, by Theorem \ref{el St general}, there exists a unique  $\nu\in\Lambda_t(\cD(\cV)\,)$  that is $\prec_w$-minimizer on this set. Moreover, recall that such vector $\nu$ is explicitly described in Proposition \ref{pdefi del nulam}. 
 \EOE
 \end{rem}

\begin{teo}[Spectral structure of Global minima in $\cD_t(\cF)$]
\label{cor nu}
Let $\RSV \in \Fnd$ with   $\la = \la(S_\cV\inv)\,$, $m  = 2\, d-n$ and 
 $t\ge \tr \la$. Let $\nu=\nu(\lambda,m,t)\in \R_{+}^d\,^\downarrow$ be as in Proposition \ref{pdefi del nulam}. Then, $\nu\in \Lambda_t(\cD(\cV)\,)$ and we have that: 
\ben 
\item 
If $\cW_t\in \cD_t(\cV)$ is such that $\lambda(S_{\cW_t})=
\nu$ then 
\begin{equation*}\label{eq hay min}
\sum_{i\in \IN{d}}f(\nu_i)=P_f(\cW_t)\leq P_f(\cW) \  \ 
\text{ for every } \cW\in \cD_t(\cV)\,, 
\end{equation*} 
and every increasing convex function $f:[0,\infty)\rightarrow [0,\infty)$.
\item 
If we assume further that $f$ is strictly convex then, for every global 
minimizer $\cW'_t$ of $P_f(\cdot)$ on $\cD_t(\cV)$
we get that $\lambda(\cW'_t)=\nu$.
\een
\end{teo}
\begin{proof}
As explained in Remark \ref{cacho prueba}
we see that $\nu \in\Lambda_t(\cD(\cV)\,) $ is such that 
$\nu\prec_w \mu$ for every $\mu\in \Lambda_t(\cD(\cV)\,)$. By the remarks in 
Section \ref{subsec 2.2} we conclude that, if $\cW_t$ is as above 
and $\cW\in \cD_t(\cV)$ then 
$$
P_f(\cW_t)=\tr(f(S_{\cW_t}))=\tr f(\nu)\leq \tr f(\lambda(S_\cW))=P_f(\cW)\ , 
$$ 
since $\lambda(S_\cW)\in\Lambda_t(\cD(\cV)\,)$.
Assume further that $f$ is strictly convex and let $\cW'_t$ be a global 
minimizer of $P_f(\cdot)$ on $\cD_t(\cV)$. Then, we have that 
$$
\nu\prec_w \lambda(S_{\cW'_t}) \peso{but} 
\tr f(\lambda(S_{\cW'_t})) 
= P_f(\cW'_t) \le  P_f(\cW_t) =  \tr \, f(\nu)\ .
$$ 
These last facts imply (see Section \ref{subsec 2.2}) that $\lambda(S_{\cW'_t})=\nu$ as desired.
\end{proof}

\pausa
Next we describe the  geometric structure of the global minimizers of the (generalized) 
frame potential $P_f(\cdot)$ in $\cD_t(\cV)$, in terms of their frame operators.

\begin{teo}[Geometric Structure of global minima in $\cD_t(\cF)$] \label{El S bis}
Let $\cV \in \RS(n \coma d)$, $m= 2d- n$, let $t\ge \tr \,S_\cV\inv \,$ and  
denote by $\la = \la(S_\cV\inv)$. Let $f:[0,\infty)\rightarrow [0,\infty)$ an increasing and strictly convex function. 
\ben
\item 
If $\cW\in \cD_t(\cF)$ is a global minimum of $P_f$ in $\cD_t(\cF)$ 
then there exists   $\{h_i\}_{i\in \IN{d}}\,$,
an ONB   of eigenvectors for $S_\cV\inv \coma \lambda$ such that
$$
S_\cW = S_\cV\inv +   \suml_{i=1}^{d-r'} \Big(\,c_{\la\coma m}( t)
-\la _{r'+i}\Big) \, 
h_{r'+i} \otimes h_{r'+i}  \ ,
$$
where $r'  = \max \{r_{\la\coma m}(t)\coma m\}$.
\item  If we further assume 
any of the conditions of item 3 of Theorem \ref{el St general}, 
there exists a unique $S_t\in\ese\cD_t(\cV)$ such that if $\cW$ is a global minimum of $P_f$ in $\cD_t(\cF)$ 
then $S_\cW=S_t$.
\een
\end{teo}
\proof 
It is a consequence of Theorems \ref{el St general} and \ref{cor nu} together with Proposition \ref{pdefi del nulam}.
\QED

\begin{rem} 
Fix $\RSV \in \Fnd$ 
and  $m = 2\,d-n$. 
Denote by $\la = \la(S_\cF\inv)$.  If $m >0$ 
then there exist $t\in \R_{>0} $
and a constant vector 
\beq\label{la1 lam}
c\, \uno_d \in  \Lambda_t(\cD(\cV)\,) 
\iff 
\la_1 = \la_m\ . 
\eeq
In this case $c = \la_1$ and $ t = d\, \la_1\,$. 
The proof uses the characterization of 
$\Lambda_t(\cD(\cV)\,)$ given in  Remark \ref {estan en U} 
and Corollary \ref{cor nuevo} (see also Definition \ref{def Lalam}). 
Indeed, if $\nu = c\, \uno_d \in  
\Lambda_t(\cD(\cV)\,) $ then, by Eq. \eqref{Lalam},  
$$
c = \nu_d \le \la_m \le \la_1 \le \nu_1 = c \implies 
\la_1 = \la_m = c\ . 
$$
Conversely, if $\la_1 = \la_m$ and $ t = d\, \la_1\,$, then 
by Corollary \ref{cor nuevo} 
it is easy to see that the vector 
$\la_1 \, \uno_d \in  \Lambda_t(\cD(\cV)\,)$. 
Therefore the frame $\cF$ has a 
dual frame which is tight \sii 
\bit 
\item $m = 2\,d-n \le 0$. Recall that in this case 
$\nu_{\la\coma m}(t) = \frac{t}{d} \cdot \uno_d$
for every $t\ge d\, \la_1\,$.  
\item  $m\in \IN{d-1}$ and $\la_{d-m+1}(S_\cF) = \la_{d}(S_\cF)$ i.e., the multiplicity of the smaller eigenvalue $\la_{d}(S_\cF)$ of $S_\cF$ is  greater or equal than $m$. 
This is a consequence of 
Eq. \eqref{la1 lam}.  
\eit
In particular, if $m\in \IN{d-1}$ then there is a Parseval dual frame for $\cF$ \sii 
$$
\la_{d-m+1}(S_\cF) = \la_{d}(S_\cF) = 1 \iff S_\cF\ge I_d \py 
\rk \, (I_d-S_\cF)  \le d-m =  \dim \ker T_\cF^*\ .
$$ 
Observe that the equivalence also holds if $2d\le n$. In this case  
there is a Parseval dual frame for $\cF \iff S_\cF\ge I_d \,$,
because the restriction $\dim \ker T_\cF^* = n-d \ge d \ge \rk \, (I_d-S_\cF)$ 
is irrelevant. This characterization  was already proved by Han  
in \cite{Han}, even for the infinite dimensional case.
 \EOE

\end{rem}

\begin{rem} \label{rem algo 3}
Using  the characterization of $\cD(\cV)$ given in the proof of Proposition \ref{los S 1} every  optimal dual frame $\cW=\{g_i\}_{i\in \IN{n}}$ is  constructed from the canonical dual of $\cV$: 
each $g_i=S_{\cV}^{-1}\, f_i+h_i\,$, for a 
 family $\cV_1=\{h_i\}_{i\in \IN{n}}$ which satisfies $T_{\cV_1}^*T_{\cV_1}=B$ and 
$T_{\cV_1}^*T_\cV=0$.
\EOE
\end{rem}

\pausa
As it was done with the completion problem, the previous results can be implemented in M{\scriptsize ATLAB} in order to construct optimal dual frames for a given one when a  a tracial condition is imposed. 

\pausa
It turns out that in this case,  once we have calculated the optimal $B$, we must improve a different type of factorization of $B$.  
Now  $B=X^*X$ should satisfy $R(T_\cV)\inc \ker X^*$. In the algorithm developed, $X^*=B\rai W^*$ where $B\rai$ has no cost of construction since we already have the eigenvectors of $S_\cF^{-1}$ . In addition $W$ is constructed using the first $d-r$ vectors of the ONB of $\ker T_{\cV}^*$ (computed with the 'null' function) and adding $r$ zero vectors in order to obtain an $n\times d$ partial isometry.

\begin{exa}\label{exa dual}
The frame operator of the following frame  $\cF\in \RS(8\coma 5)$ has eigenvalues listed by $\la=(\frac{5}{2}\coma 2\coma \frac{2}{3}\coma\frac{1}{3}\coma \frac{1}{4})$:
\[
T_{\cF}^* = {\scriptsize  
\left[
 \begin{array}{rrrrrrrr}
   -0.5124&    0.5695&    0.4542&   -0.3527&   -0.2452&    0.1260&    0.0558&   -0.3513\\
   -0.4965&    0.0478&    0.1579&   -0.2299&   -0.9348&   -0.6935&   -0.0836&    0.7641\\
    0.2777&    0.2875&   -0.4974&    0.0086&    0.1893&   -0.0916&    0.2501&   -0.0722\\
   -0.3793&   -0.7849&   -0.4783&   -0.2566&    0.3450&   -0.0749&   -0.2939&    0.3785\\
    0.0725&   -0.0803&   -0.2075&   -0.2967&   -0.1518&    0.2077&   -0.2050&    0.4226\\

 \end{array}
 \right]}
 \ .\]
Therefore, $\la=\la(S_\cV\inv)=(4\coma 3\coma \frac{3}{2}\coma \frac{1}{2}\coma \frac{2}{5})$ and $\tr \,S_\cV\inv \,=9.4$. We also have that $m=2d-n=2$. Consider $t=16.5$, then an optimal dual $\cW\in \cD_t(\cF)$ for $\cF$ is given by
\[
T_{\cW}^* = {\scriptsize  
\left[
 \begin{array}{rrrrrrrr}
   -1.0236&    0.2319&   -0.0181&   -0.5802&    0.0438&    0.7316&    1.0846&   -0.0143\\
   -0.2583&    0.6148&    0.5219&    0.1585&    0.6493&   -0.7116&    0.2109&    1.2138\\
    0.3080&    0.6525&   -1.0323&   -0.8031&    0.2306&   -0.8740&    0.2856&   -0.3488\\
   -1.1868&    0.1198&   -0.8331&    0.4816&    0.4222&   -0.0495&   -0.8551&   -0.0836\\
    0.5506&    0.5428&   -0.2035&   -0.5871&   -0.2309&    1.1268&   -0.7891&    0.8432\\

     \end{array}
 \right]}
 \ .\]
 Here, the optimal spectrum $\nu=\nu_{\la \coma m}(16.5)$ is given by 
 $\nu=(4\coma 3.166\coma 3.166\coma 3.166 \coma 3)$.
 \EOE
\end{exa}

\section{APPENDIX - Proof of Theorem \ref{el St general}}\label{appendixiti}

In this section we obtain the proofs of Theorem \ref{el St general} and Proposition \ref{pdefi del nulam} stated in Section \ref{unified} in a series of steps. In the first step we introduce the set $U(S_0,m):=U_{\tr(S_0)}(S_0,m)$ and characterize its spectral picture $\Lambda(U(S_0,m))$ - i.e. the subset of $\R_{+}^d\, ^\downarrow$ of eigenvalues $\lambda(S)$, for $S\in U(S_0,m)$ - in terms of the so-called Fan-Pall inequalities. In the second step we show the existence of a $\prec_w$-minimizer within the set $\Lambda(U_t(S_0,m))$ and give an explicit (algorithmic) expression for this vector. Finally, in the third step we characterize the geometrical structure of the positive operators $S\in U_t(S_0,m)$ such that $\lambda(S)$ are $\prec_w$-minimizers within the set $\Lambda(U_t(S_0,m))$, in terms of the relation between the eigenspaces of $S$ and the eigenspaces of $S_0$. It is worth pointing out that the arguments in this section are constructive, and lead to algorithms that allow to effectively compute all the parameters involved.

\section*{Step 1: spectral picture of $U(S_0,m)$}\label{sec EP}
Recall that $\R_{+}^d\, ^\downarrow$ is the 
set of vectors $\mu \in \R_+^d$ with non negative and decreasing entries (i.e. $\mu\in \R_{+}^d$ with $\mu^\downarrow=\mu$);
also, given $S\in \matrec{d}^+$, $\la(S) \in 
\R_{+}^d\,^\downarrow$ denotes the 
vector of eigenvalues of $S$ - counting multiplicities - and arranged in decreasing order. 

Given $S_0\in \matpos$, $m<d$ and integer and $t\geq \tr(S_0)$ then in Eq. \eqref{defi usubt} we introduced
$U_t(S_0\coma m)=\{S_0+B:\   B\in \matpos \, , \ \rk \, B 
\le d-m \ , \ \tr(S_0+B)\geq t\ \}$. In this section we consider $$U(S_0,m):=U_{\tr(S_0)}(S_0,m)= \{S_0+B:\   B\in \matpos \, , \ \rk \, B 
\le d-m \ \}\ $$ together with its spectral picture $\Lambda(U(S_0,m)):=\Lambda(U_{\tr(S_0)}(S_0,m))$ (see Eq. \eqref{defi imag espec usubt} in Remark \ref{rem imag espec usubt}).
We shall also use the following notations: 
\ben
\item Given $x\in \C^d$ then $D(x) \in \mat $ denotes the 
diagonal matrix with main diagonal $x$. 
\item If $d\le n$ and $y \in \C^d$, we write $ (y\coma 0_{n-d})\in \cene$, where  
$0_{n-d}$ is the zero vector of $\C^{n-d}$. 
In this case, we denote by $D_n(y) = D\big( \, (y\coma 0_{n-d})\,\big) \in \matrec{n}$. 
\een

\begin{teo} \label{TP}
Let $S_0\in\matpos$, $m<d$ be an integer and $\mu\in \R_{+}^d\, ^\downarrow$. 
Then the following conditions are equivalent: 
\ben
\item\label{cond1} 
There exists $S\in U(S_0\coma m)$ such that $\lambda(S)=\mu$. 
\item\label{cond2} There exists an orthogonal projection 
$P\in \matrec{2d-m} $ such that $\rk \,P=d$ and  
\beq\label{item2}
\lambda\left( P\, D_{2d-m}(\mu )\, P\right)= 
\big(\, \lambda(S_0)\coma 0_{d-m}\,\big) \ .
\eeq
\een
\end{teo}

\proof

$1 \Rightarrow 2$. Let $B\in\matpos$ be such that $\rk(B)\leq d-m$ and $\lambda(S_0+B)=\mu$. 
Thus, $B$ can be factorized as $B=V^*V$ for some $V\in \cM_{d-m,\,d}(\C)$. 
If 
\beq\label{reveq1} 
T=\begin{pmatrix}\ S_0\rai\\ V\end{pmatrix}\in\cM_{2d-m,\,d}(\C)
 \ \ \Rightarrow \ \ T^*T=S_0+B \py 
 TT^*=\begin{pmatrix}S_0 & S_0^{1/2}V^* \\ V\,S_0^{1/2} & V\,V^*\end{pmatrix}\  .
\eeq
Let $U\in \cU(2d-m)$ be such that $U(TT^*)U^*=D(\lambda(TT^*))=D_{2d-m}(\mu)$ and let 
$P\in \cM_{2d-m }(\C)$ be given by $P=U\,P_1\,U^*$, where $P_1=I_d\oplus 0_{d-m}\,$. 
Notice that, by construction, $P$ is an orthogonal projection with $\rk \,P=d$ and, 
by the previous facts, 
$$ 
P\,D_{2d-m}(\mu)\,P= U \,P_1\,(TT^*)\,P_1\,U^*=
U \begin{pmatrix}S_0& 0\\ 0 & 0 \end{pmatrix}U^*\ ,
$$ 
which shows that Eq. \eqref{item2} holds in this case.

\pausa
$2 \Rightarrow 1$. Let $P\in \cM_{2d-m}(\C)$ be a 
projection as in item 2. 
Then, there exists $U\in \cU(2d-m)$ such that $U^*P\,U=P_1\,$, 
where $P_1=I_d\oplus \,0_{d-m}$ as before. Hence, we get that 
\beq\label{reveq2}
\lambda (P_1 (U^*D_{2d-m}(\mu)\,U)\,P_1)=\lambda(S_0,0_{d-m})\ .
\eeq  
Since $\rk(U^* D_{2d-m}(\mu)\,U)\leq d$ then we see that there exist 
$T\in \cM_{2d-m,\,d}(\C)$  such that $U^*D_{2d-m}(\mu)\,U=TT^*$. 
Let $T_1\in \matpos$ and  $T_2\in \cM_{d-m,d}(\C)$ such that 
$$
T=\begin{pmatrix} T_1\\T_2 \end{pmatrix}  \implies 
U^*D_{2d-m}(\mu)\,U=TT^*
=\begin{pmatrix} T_1T_1^* & T_1T_2^*\\ T_2T_1^*&T_2T_2^*\end{pmatrix}\ .
$$ 
Then   $\lambda(T_1T_1^*)=\lambda(S_0)$ by Eq. \eqref{reveq1}. 
On the other hand, notice 
that $\lambda(T^*T)=\mu$ and 
$$
T^*T=T_1^*T_1+T_2^*T_2 \igdef S_1+B_1
\peso{with} \lambda(S_1)=\lambda(T_1T_1^*)=\lambda(S_0)  \py\rk(B_1)\leq d-m \ .
$$ 
Let $W\in \cU(d)$ such that $W^*S_1W=S_0\,$. Then $S \igdef W^*(T^*T)W=S_0+B$ 
satisfies that  $\lambda(S)=\mu$ and $\rk(B)=\rk(W^*B_1W)\leq d-m$. 
Then $\mu = \lambda(T^*T)=\la(S)\in \Lambda(U(S_0\coma m)\,)$. 
\QED

\begin{rem} \label{FP}

Let $S_0\in\matpos$, $m<d$ be an integer and $\mu\in \R_{+}^d\, ^\downarrow$ 
as in Theorem  \ref{TP}. It turns out that condition \eqref{item2} can be characterized in terms of 
interlacing inequalities. 

\pausa
More explicitly, given $\mu\in \R_{+}^d\, ^\downarrow\,$, 
by the Fan-Pall inequalities (see \cite {LM}),  the existence of a projection $P\in \matrec{2d-m}$ 
satisfying \eqref{item2} for $\mu$ 
is equivalent to the following inequalities: 
\begin{enumerate}
\item 
$\mu \geqp \lambda(S_0)$, i.e. 
$\mu_i \ge \lambda_{i}(S_0)$ for every  $i \in \IN{d}\,$. 
\item If $m\geq 1$ 
then  $\mu$ also satisfies 
$$
\mu_{d-m+i}  \le \la_i (S_0 )
\peso{for every} i \in \IN{m}  
\ ,
$$
where the last inequalities compare the first $m$ entries of $\la(S_0 )$ 
with the last $m$ of $\mu$. 
\end{enumerate}
These facts together with Theorem \ref{TP} give a complete description 
of the spectral picture  of the set $U(S_0 \coma m)$,
 which we write as follows. 
\EOE
\end{rem}

\begin{cor}\label{desigualdades de espectros de duales}
Let $S_0\in\matpos$ and $m<d$ be an integer.
Then, the set $\Lambda(U(S_0\coma m))$ can be characterized as follows: 
\ben 
\item If $m\leq 0$, we have that 
\beq\label{mayor 2d}
\mu  \in \Lambda(U(S_0\coma m)) \iff \mu \geqp \lambda(S_0 )\ .
\eeq
\item If $m\geq 1$, then 
\beq\label{menor 2d}
\mu \in  \Lambda(U(S_0\coma m)) \iff \mu \geqp \la (S_0 ) \py 
\mu_{d-m+i}  \le \la_i (S_0)
\peso{for} i \in \IN{m}  \ .
\eeq
\een
\end{cor}
\proof 
It follows from Theorem \ref{TP} and the 
Fan-Pall inequalities of 
Remark \ref{FP}. \QED

\begin{cor}\label{La conv} 
Let $S_0\in\matpos$ and $m<d$ be an integer.
  Then $ \Lambda(U(S_0\coma m))$ is convex.
\end{cor}
\proof It is clear that the inequalities given in Eqs. \eqref{mayor 2d} and \eqref{menor 2d}
are preserved by convex combinations. Observe that also the set $\R_{+}^d\,^\downarrow $
is convex. \QED

\begin{rem}\label{los espectros no determinan} 
Let $S_0\in\matpos$, $m<d$ be an integer and $S\in \matpos$. The reader should note that the fact that $\lambda(S)\in \Lambda(U(S_0\coma m))$ does not imply that $S\in U(S_0\coma m)$. Indeed, it is fairly easy to produce examples of this phenomenon. Therefore, the spectral picture of $\Lambda(U(S_0\coma m))$ does not determine the set $U(S_0\coma m)$. This last assertion is a consequence of the fact that $U(S_0\coma m)$ is not saturated by unitary equivalence. Nevertheless, $\Lambda(U(S_0\coma m))$ allows to compute minimizers of submajorization in $U(S_0\coma m)$, since submajorization is an spectral preorder.
\end{rem}

\section*{Step 2: minimizers for submajorization in $\Lambda(U_t(S_0,m))$}\label{mayo vec}
The spectral picture of $U(S_0\coma m)$ studied in the previous section 
motivates the definition of the following sets.

\begin{fed}\label{def Lalam}\rm 
Let $\lambda\in \R_{+}^d\,^\downarrow$ 
and take an integer $m<d$.     
We consider the set 
\beq\label{Lalam}
\Lambda(\la \coma m)=
\begin{cases} \ 
 \big\{\mu\in \R_{+}^d\,^\downarrow:
 \ \mu\geqp \lambda\,  \big\}  &  \mbox{if} \ \ m\le 0 \\&\\
 \ \big\{\mu\in \Lambda(\la \coma 0) : 
\mu_{d-m+i}\le \lambda_{i} \peso{for every} i \in \IN{m} \big\}
&  \mbox{if} \ \ m\ge1  \ .\end{cases}
\eeq
Denote by $t_0= \tr \lambda$. For $t\ge t_0\,$, we 
also consider the set  
\beq
\Lambda_t(\la \coma m)=\big\{\mu \in \Lambda(\la \coma m): 
\ \tr \mu\geq t \} \ . 
\EOEP
\eeq
\end{fed}

\pausa
Now Corollary \ref{desigualdades de espectros de duales} 
can be rewritten as

\begin{cor} \label{cor nuevo} \rm
Let $S_0\in\matpos$ with $\lambda(S_0)=\lambda$, $m<d$ be an integer and $t\geq \tr(\lambda)$.
Then we have the identities
$ \Lambda(U(S_0\coma m))= \Lambda(\la\coma m)$ and $ \Lambda(U_t(S_0\coma m))= \Lambda_t(\la\coma m)$  .  
\QED
\end{cor}

\pausa
In this section, as a second step towards the proof of Theorem \ref{el St general}, we show that the sets $\Lambda_t(\la \coma m)$ have minimal elements with respect to submajorization and we describe explicitly these elements.

\pausa
Let $\lambda\in \R_{+}^d\,^\downarrow$ and $t_0 = \tr \, \la$. 
We recall the maps $r_\lambda(\cdot)$ and $c_\lambda(\cdot)$ introduced in \ref{unified}.  
Fix $t\ge t_0\,$. 
Then 
\ben 
\item Given $r \in \IN{d-1}\cup\trivial$ 
we denote by  $p_\la(r\coma t) = \ \frac{ t-\sum_{j=1}^r\lambda_j }{d-r}$, 
where we set $\sum_{j=1}^{0}\lambda_j =0$.
\item The maps $r_\la:[t_0\coma +\infty)\rightarrow \IN{d-1}\cup\trivial$
and $c_\la:[t_0\coma +\infty)\rightarrow \R_{\geq 0}$ given by 
\beq\label{rla cla bis}
r_\la(t)  =\min \{r\in\IN{d-1}\cup\trivial : p_\la(r\coma t) 
\geq \lambda_{r+1}\} \py 
c_\la(t) = 
\frac{t-\sum_{i=1}^{\,r_\la(t)}\lambda_i}{d-r_\la(t)}  \  .
\eeq
\een
In the following Lemma we state several properties of these maps, which we shall use below. The proofs are technical but elementary, 
so that we only sketch the essential arguments. 

\begin{lem}\label{r y c}
Let $\lambda\in \R_{+}^d\,^\downarrow$  and $ t_0 = \tr \lambda$. 
\ben
\item \label{item1}The function $r_\la$ is non-increasing and right-continuous,  
with $\lambda_{r_\la(t_0)+1}=\la_d\,$.
\item The image of $r_\la$ is the set $\cB= \{ k \in \IN{d-1}: \la_k > \la_{k+1}\}\cup\trivial$. 
\item \label{item22} The map $c_\la$ is piece-wise linear, strictly increasing and continuous.
\item \label{item 33}We have that  $c_\la(t_0)=\lambda_d\,$ and 
$c_\la(t) = t/d$ for $t\ge d\, \la_1\,$. 
\item For every $t\in [t_0\coma d\,\la_1)$, if $r = r_\la (t)$ then $\la _{r+1}\le c_\la(t) < \la_r\,$. 
In other words
\beq\label{circular}
r_\la(t) = \min \,\{ r \in \IN{d-1}\cup \trivial \, : \, \la_{r+1}\le c_\la (t) \, \} \ .
\eeq
\item For any $k \in \cB$ let $s_k = \suml_{i=1}^{k} \la_i + (d-k\,) \, \la_{k+1}\,$. 
Then $r_\la(s_k) = k$ and $c_\la(s_k) = \la_{k+1}\,$. 
Moreover, the set 
$\cA$ of discontinuity points of $r_\la$ satisfies that  \rm
$$
\cA=  \{t\in (t_0\coma +\infty):\ c_\la(t)=\lambda_{r_\la(t) +1}\} 
= c_\la\inv \{\la_i:\ \la_i\neq \lambda_d \}=  \{s_k : k \in \cB\}\ .
$$ \it
\item Given   $t\in [t_0\coma +\infty)$,  such that $c_\la(t) = \la_m\,$
(even if $m\notin \cB$), then 
\bit
\item $t\in \cA \iff \la_m \neq \la_d\,$.
\item $r_\la(t)=0 \iff c_\la(t)=\la_1 
\iff t=d\,\la_1\,$.  
\item If $\la_m \neq\la_1\,$, then $r_\la(t)  
= \max\{j\in \IN{d}: \la_j > \la_m \}$ and 
\beq\label{el la + 1}
t = \suml_{i=1}^{m} \la_i + (d-m) \, \la_{m} = 
\suml_{i=1}^{r_\la(t)} \la_i + (d-r_\la(t)\,) \, \la_{m}\ .
\eeq

\eit
\een
\end{lem}

\begin{proof}
Given  $t\in [t_0\coma d\,\la_1)$ and $1\leq r\leq d-1$,  then $r=r_\la(t)$ if and only if 
\beq\label{ecu1}
c_\la(t)=p_\lambda(r \coma t)\geq \lambda_{r+1} \peso{ and } p_\lambda(r-1 \coma t)<\lambda_r\ .
\eeq
On the other hand  the map  $t \mapsto p_\la(r,t)$ is linear, continuous and increasing 
for any $r$ fixed. 
From these facts one easily deduces the right continuity of the map $r_\la\,$, and that 
the map $c_\la$ is continuous at the points where $r_\la$ is. 
We can also deduce that if $c_\la(t)\neq \la_{r_\la(t)+1}$ then $r_\la$ 
is continuous (i.e. constant) near the point $t$. Observe that, if $r =r_\la(t)$, then 
\beq\label{sube}
\la_r \stackrel{\eqref{ecu1}}{>} p_\lambda(r-1 \coma t) = \frac {(d-r)p_\lambda(r\coma t)+ \la_r}{d-r+1} 
\implies \la_r> p_\lambda(r\coma t) \ge \la_{r+1}  \implies r\in \cB\ .
\eeq
Using that $r_\la(t) = 0$ for $t\ge d\, \la_1\,$, that 
$c_\la(t_0) = \la_d\,$, and the right continuity of the map $r_\la\,$,
we have that   
$\cA =\{t\in (t_0\coma +\infty):\ c_\la(t)=\lambda_{r_\la(t)+1}\} 
= c_\la\inv \{\la_i:\ \la_i\neq \lambda_d \}$. 

\pausa
Hence, in order to check the continuity of $c_\la$ we have to verify the continuity 
of $c_\la$ from the left at the points $t> t_0$ for which $c_\la(t)=\lambda_{r_\la(t)+1}\,$. 
Note that, if $r =  r_\la(t)$, then $r\in \cB$ and
\beq\label{ecu1.5}
c_\la(t) = p_\lambda(r\coma t)=\frac{ t-\sum_{j=1}^{r}\lambda_j }{d-r} 
= \la_{r+1} \implies t=\sum_{j=1}^{r}\lambda_j+(d-r)\lambda_{r+1}\ .
\eeq 
If $c_\la(t)=\lambda_d\,$ then $t = t_0$ and there is nothing to prove. 
Assume that $c_\la(t)=\lambda_{r_\la(t)+1}>\lambda_d\,$.
Then $\hat r=\max\{j \in \IN{d-1}:\ \lambda_j=\lambda_{r+1}\}$
is the first element of  $\cB$ after $r$. 
Note that $\lambda_{\hat r+1}<\lambda_{\hat r}= \lambda_{r+1}\,$. 
We shall see that if $s<t$ near $t$, then $r_\la(s) = \hat r$. 
Indeed, as in  Eq. \eqref{ecu1.5}, 
$$
p_\lambda(\hat r \coma t+x)= 
\frac{ (d-r)\lambda_{r+1} - 
\sum_{j=r+1}^{\, \hat r}\lambda_j +x}{d-\hat r}
=\lambda_{r+1}+\frac{x}{d-\hat r}>\lambda_{\hat r+1} 
\py
$$ 
$$
p_\lambda(\hat r-1 \coma t+x)= \frac{ (d-r)\lambda_{r+1} - 
\sum_{j=r+1}^{\,\hat r-1}\lambda_j +x}{d-\hat r+1}
=\lambda_{r+1}+\frac{x}{d-\hat r +1}<\lambda_{r+1}=\lambda_{\hat r}\ . 
$$ 
for $x\in (-\varepsilon,0]$ if 
$\varepsilon>0$ sufficiently small. 
By Eq. \eqref{ecu1} we deduce that  $r_\la(t+x)=\hat r\neq r_\la(t)$ for 
such an $x$, so that $t \in \cA$ ($r_\la$ is discontinuous at $t$). 
On the other hand,  
$$
c_\la(t+x) =p_\lambda(\hat r,t+x)
=\lambda_{r_\la(t) + 1}+ \frac{x}{d-\hat r}\implies \lim_{x\rightarrow 0^-}c_\la(t+x)
=\lambda_{r_\la(t)+1} = c_\la(t)\ .
$$
This last fact implies that $c_\la$ is continuous and, since $r_\la$ 
is right-continuous, that $c_\la$ is a piece-wise linear and 
strictly increasing function. 
With the previous remarks, 
the proof of all other statements of the lemma  
becomes now straightforward. 
\end{proof}

\begin{num}\label{rlam clam bis}
Fix $\lambda\in \R_{+}^d\,^\downarrow$. Take an integer  $m<d$. Recall that 
if $m>0$ we denote by 
$$
s^*=s^*(\la\coma m) =  c_\la\inv (\la_m) =  \suml_{i=1}^{m} \la_i + (d-m) \, \la_{m} \ . 
$$ 
Now we rewrite the definition of the maps $r_{\la\coma m}$ and $c_{\la\coma m}\,$: 
If $m>0$ 
and $t\in [t_0 \coma +\infty)$ let
$$
c_{\la\coma m}(t) \igdef \begin{cases} \ c_\la(t) & \mbox{if} \ \ t\le s^*  \\
\la_m+\frac{t-s^*}{d-m} & \mbox{if} \ \ t> s^* 
\end{cases} \quad \quad \py
$$
$$
r_{\la\coma m}(t) \igdef 
\min\limits \{r\in\IN{d-1} \cup \trivial : c_{\la\coma m}(t)  
\geq \lambda_{r+1}\}
  \ \ .
$$
If $m\le 0$ and $t\in [t_0 \coma +\infty)$ we define $c_{\la\coma m}(t) = c_{\la}(t) $ 
and $r_{\la\coma m}(t) = r_{\la}(t) $.  
Note that, by Eq. \eqref{circular}, $r_{\la\coma m}(t)= r_{\la}(t)$ for every $t\le s^*\,$. \EOE
\end{num}

\begin{cor} 
Let $\lambda\in \R_{+}^d\,^\downarrow\,$ and fix an  integer $m<d$. 
Then the map $r_{\la\coma m}$ is not increasing and right continuous 
and the map $c_{\la\coma m}$ is strictly increasing and continuous 
on $[\tr \, \la \coma +\infty)$. 
\end{cor}
\proof
The mentioned properties of the map $c_{\la\coma m}$ were proved 
in Remark \ref{clm y cl}  (whose proof uses Lemma \ref{r y c}). 
With respect to the map $r_{\la\coma m}\,$, the statement follows from 
Lemma \ref{r y c} and \ref{rlam clam bis}.

\subsection*{Minimizers for submajorization in 
$\Lambda_t(\lambda \coma m)$ for $m\le 0$.}
The following Lemma is a standard fact in majorization theory.
We include a short proof of it for the sake of completeness. 

\begin{lem}\label{lem may} \rm
Let $\alpha \coma \gamma \in \R^p$, 
$\beta\in \R^q$ and $x\in \R$ 
such that $x \le \min_{k \in \IN{p}} \gamma_k\,$.
Then, 
$$
\tr\, (\gamma \coma b\,\uno_q) \, \le \, \tr\,(\alpha \coma \beta) 
\peso{and} \gamma \, \prec_w \, \alpha \ 
\implies \ (\gamma \coma x\,\uno_q) \, \prec_w \, (\alpha \coma \beta) \ . 
$$
Observe that we are not assuming that $(\alpha \coma \beta) = 
(\alpha \coma \beta)^\downarrow$. 
\end{lem}
\proof Let $h = \tr\, \beta $ and 
$\rho = \frac hq \, \uno_q\,$. 
Then it is easy to see that 
$$
\barr{rl}
\suml_{i\in \IN{k}} (\gamma^\downarrow \coma x\,\uno_q)_i &\le \ 
\suml_{i\in \IN{k}} (\alpha^\downarrow \coma \rho)_i  \ \le \  
\suml_{i\in \IN{k}}(\alpha^\downarrow \coma \beta^\downarrow) _i 
\peso{for every} k \in \IN{p+q}\ .
\earr
$$ 
Since  $(\gamma^\downarrow \coma x\,\uno_q) = 
(\gamma \coma x\,\uno_q)^\downarrow$, 
 we can conclude that 
$(\gamma \coma x\,\uno_q)  \prec_w  (\alpha \coma \beta)$. 
 \QED

\pausa 
In the following statement we shall use the maps $r_\la$ and $c_\la$ 
defined in  Eq. \eqref{rla cla bis} (or Definition \ref{defi r c}).

\begin{teo}\label{prop consec de la defi}
Fix $m\le 0$. Let $\lambda\in \R_{+}^d\,^\downarrow$, 
$ t_0= \tr \lambda $  and $t\in [t_0\coma +\infty)$. 
Consider  the vector 
\beq\label{el v}
\nu = \nu_\la (t) \igdef \big( \la_1 \coma \dots\coma \la_{r_\la(t)}\coma c_\la(t)\coma \dots 
\coma c_\la(t)\, \big) 
 \peso{if} r_\la(t)>0 \ , 
\eeq
or $\nu = \frac td \ \uno_d = c_t(\la) \, \uno_d
\in\Lambda_t(\lambda \coma m)$ if $r_\la(t)=0$. Then $\nu$ 
satisfies that 
\beq\label{v cumple 1}
\nu\in\Lambda_t(\lambda\coma m )\ , \quad 
\tr \nu=t \py \nu\prec_w \mu  \peso{for every} \mu \in
\Lambda_t(\lambda\coma m) \ .
\eeq
\end{teo}

\proof 
Given $t\in [t_0\coma +\infty)$, we denote by $r = r_\la(t)$. If $r =0$ then, 
$$
t\ge d\, \la_1 \py \la = \la^\downarrow
 \implies  c_\la(t) = \frac td \ge \la_1 
{\implies} \nu = c\, \uno_d \in  \Lambda_t(\la\coma m) \ . 
$$
It is clear that such a vector 
must satisfy that $\nu \prec_w \mu$ for every $\mu \in 
\Lambda_t(\la\coma m)$.

\pausa
Suppose now that $r \ge 1$, so that $t< d\, \la_1\,$. 
Recall from Lemma \ref{r y c} that in this case we have that 
$\la _{r+1}\le c_\la(t) < \la_r\,$. Hence 
$\nu \geqp \la$ and $\nu = \nu^\downarrow$. 
It is clear from Eq. \eqref{defi cs} that
$\tr(\nu)=t$. 
From these facts we can conclude that 
$\nu\in \Lambda_t(\la\coma m)$ as claimed. 

\pausa
Now let $\mu \in \Lambda_t(\la\coma m)$ and notice that, 
since $\mu \geqp \la$, we get that 
$$
\sum_{i=1}^k\mu_i\geq \sum_{i=1}^k\lambda_i=\sum_{i=1}^k\nu_i
\peso{for every} 1\leq k\leq r_\la(t) \ .
$$ 
Now we can apply Lemma \ref{lem may} 
(with $p = r_\la(t) \,$ and $x=c_\la(t) \,$) and deduce that 
$\nu \prec_w \mu$.\QED

\subsection*{Minimizers for submajorization in 
$\Lambda_t(\lambda \coma m)$. The general case.}
Recall that $
\Lambda_t(\la \coma m)=\big\{\mu\in \R_{+}^d\,^\downarrow:\ \mu\geqp \lambda\, ,
\ \tr \mu\geq t $ and $ \mu_{d-m+i}\le \lambda_{i} $ for every 
$i \in \IN{m} \big\}$, for each $m\in \IN{d-1}\,$.  
In what follows we shall compute 
a minimal element in $\Lambda_t(\la\coma m)$ 
with respect to submajorization in terms of 
the number $s^*=s^*(\la\coma m) \igdef c_\la\inv (\la_m)$ and 
the maps $r_{\la\coma m}$ and $c_{\la\coma m}$ 
described in Definition \ref{defi r c} (see also \ref{rlam clam bis}).

\pausa
\begin{pro} \label{<s}
Let $\lambda\in \R_{+}^d\,^\downarrow$, 
$ t_0= \tr \lambda $, $m\in \IN{d}\,$. If 
 $t\in [t_0\coma s^*(\la\coma m)]$,  
then the vector 
$ \nu = \big( \la_1 \coma \dots\coma \la_{r_\la(t)}\coma 
c_\la(t) \coma \dots \coma c_\la(t) \, \big)$ of Eq. \eqref{el v} 
satisfies that $\nu \in\Lambda_t(\lambda \coma m)$. Hence 
$$
\tr \nu=t  \ , \quad \nu_d = c_\la(t)   \peso{and}  
\nu\prec_w \mu  \peso{for every} \mu \in\Lambda_t(\lambda\coma m) \ .
$$
\end{pro}
\proof
We already know by Theorem 
\ref{prop consec de la defi} that 
$\nu\in  \Lambda_t(\lambda \coma 0)$ and $\tr \nu=t$.  
Using the inequality   $c_\la(t)  \le c_\la(s^*) 
= \la_m\,$, the verification 
of the fact that 
$\nu\in  \Lambda_t(\lambda\coma m)$ is direct. 
By Theorem \ref{prop consec de la defi}, we conclude 
that $\nu\prec_w \mu$ for every $\mu\in \Lambda_t(\la\coma m)
\inc \Lambda_t(\lambda \coma 0)$.\QED

\pausa
Recall the number 
$s^{**} = c_{\la\coma m}\inv(\la_1)= (d-m)\, \la_1 + \suml_{j=1}^m \la_j \ge s^* $ (with equality$\iff \la_1 = \la_m$)  
defined in Eq. 
\eqref{defi s**} (see also Remark \ref{clm y cl}).

\begin{fed}\label{defi del nulam} \rm 
Let $\lambda\in \R_{+}^d\,^\downarrow$, 
$ t_0= \tr \lambda $ and $m\in \mathbb{Z}$ such that $m<d$.  
Fix $t \in [t_0 \coma+\infty)$ and denote by 
$r = r_{\la\coma m}(t)\,$. 
Consider  the vector $\vlm(t)\in \R_{+}^d\,$
given by the following rule: 
\bit
\item 
If $m\le 0$ then 
$\vlm( t) = \nu_\la(t) 
\stackrel{\eqref{el v}}{=} \big( \la_1\coma \dots \coma \la_{r} \coma 
c_{\la\coma m}(t) \, \uno_{d-r}\big)$.   
\eit
If $m\ge 1$ we define 
\bit
\item 
$\vlm( t) = \big( \la_1\coma \dots \coma \la_{r} \coma 
c_{\la\coma m}(t) \, \uno_{d-r}\big)$ for $t\le s^*$ 
(so that  $r\ge m$ and $c_{\la\coma m}(t)\le \la_m$). 
\item 
$\vlm( t) =  \Big( \la_1\coma \dots \coma \la_{r} 
\coma c_{\la\coma m}(t) \, \uno_{d-m}\coma \la_{r+1}\coma 
\dots \coma \la_m\Big) $
for $t\in (s^*\coma s^{**})$, and  
\item 
$\vlm( t)  = \big( c_{\la\coma m}(t) \, \uno_{d-m}\coma 
\la_{1}\coma \dots \coma \la_m\big)$ 
for $t\ge s^{**}\,$.
\eit
If $\la_1=\la_m\,$, the second case of the definition 
of $\nu_{\la\coma m}(t)$  disappears.
\EOE
\end{fed}

\pausa
In the following Lemma we state several 
properties of the map $\vlm (\cdot)$, which are easy to see:

\begin{lem} \label{las propos del nulam} \rm
Let $\lambda\in \R_{+}^d\,^\downarrow$, 
and $m\in \mathbb{Z}$ such that $m<d$.  The map  $\vlm (\cdot)$ of Definition \ref{defi del nulam} 
has the following properties:
\ben
\item By Remark \ref{clm y cl} the vector $\vlm(t)\in \R_{+}^d\,^\downarrow \,$  
(i.e. it is decreasing) for every $t$.  
\item The map  $\vlm (\cdot)$ is continuous.
\item It is increasing in the sense that 
$t_1 <t_2\implies \vlm(t_1) \leqp \vlm(t_2)\,$.
\item More precisely, for any fixed $k \in \IN{d}\,$, the 
$k$-th entry $\vlm^{(k)}(t)$ of $\vlm(t)$ is given by 
$$
\vlm^{(k)}(t) = 
\begin{cases}  \max\,\{\la_k \coma c_{\la\coma m}(t) \} & 
\peso {if} k \le d-m \, , \\&\\
\min\, \Big\{ \, \max\,\{\la_k \coma c_{\la\coma m}(t) \} 
\coma \la_i \Big\} & \peso{if} k =  d-m+i  \, , \ \ i \in \IN{m}\ .
\end{cases}
$$
\item The vector $\vlm(t)\in\Lambda_t(\lambda\coma m)$ and  
$\tr \vlm(t) = t $ for every 
$t\in [t_0 \coma+\infty) $. \QED
\een
\end{lem}

\pausa
We can now state the main result of this section.

\begin{teo}\label{tutti nu}
Let $\lambda\in \R_{+}^d\,^\downarrow$, 
$ t_0= \tr \lambda $ and $t\in [t_0 \coma+\infty)$.
Fix $m\in \mathbb{Z}$ such that $m<d$. 
Then the vector $\vlm(t) $ defined in  \ref{defi del nulam}
is the unique element of 
$\Lambda_t(\lambda\coma m)$ such that  
\beq\label{mino a mu}
\vlm(t) \prec_w \mu  \peso{for every} \mu \in\Lambda_t(\lambda\coma m) \ .
\eeq
\end{teo}
\proof 
If $m\le 0$ the result follows from Theorem 
\ref{prop consec de la defi}. 
Suppose now that $m\ge1$. By Lemma \ref{las propos del nulam}, 
the vector $\vlm(t)\in\Lambda_t(\lambda\coma m)$ and  
$\tr \vlm(t) = t $ for  
$t\in [t_0 \coma+\infty) $. 
In Proposition \ref{<s} we have shown that $\vlm(t)$ satisfies 
\eqref{mino a mu} for every 
 $t\in [t_0\coma s^*(\la\coma m)\,]$. Hence we check the other two cases:

\pausa
Case $t\in ( s^*\coma s^{**})$: fix $\mu \in \Lambda_t(\lambda\coma m)$ 
such that $\tr \mu =t$. Let us denote by $r = r_{\la\coma m}( t)$, 
$$
\al = (\mu_1 \coma \dots \coma \mu_{r}) \ , \ \ 
\beta = (\mu_{r+1} \coma \dots \coma \mu_{r+d-m}) \ , \ \  
\gamma = (\mu_{r+d-m+1} \coma \dots \coma \mu_d) \ ,
$$  
$\rho = (\la_1 \coma \dots \coma \la_{r})$ and 
$\omega = (\la_{r+1} \coma \dots \coma \la_m)$. 
Then 
$$
\mu = (\al \coma \beta\coma \gamma) \py 
\vlm(t) = (\rho  \coma c_{\la\coma m}(t) \, \uno_{d-m}\coma  \omega) \ .
$$
Since $\mu \in \Lambda_t(\lambda\coma m)$ and 
$\tr \, \vlm(t) = \tr \, \mu  = t$, then 
$$
\rho \leqp \al \quad ,  \quad  \gamma \leqp \omega
\py \tr \,(\al\coma \beta) \ge \tr (\rho \coma c_{\la\coma m}(t) \, \uno_{d-m})
\ .
$$   
Then we can apply Lemma \ref{lem may} to 
deduce that $(\rho \coma c_{\la\coma m}(t) \, \uno_{d-m})
\prec_w (\al\coma \beta)$. Using this fact jointly 
with $\gamma \leqp \omega$ 
one easily deduces that $\vlm(t) \prec \mu$ 
(because $\tr \mu = \tr \vlm(t) = t$). 

\pausa 
The case  $t \ge s^{**}$ for vectors $\mu \in \Lambda_t(\lambda\coma m)$  
such that $ \tr \mu = t$ follows similarly. 

\pausa
If we have that  $\mu \in \Lambda_t(\lambda\coma m)$  with 
$\tr \mu = a > t$, then 
$$
\mu \in \Lambda_a(\lambda\coma m) \implies 
\vlm(t) \leqp \vlm(a) \prec \mu \implies 
\vlm(t) \prec_w \mu  \  , 
$$
where the first inequality follows from Lemma \ref{las propos del nulam}. 
\QED

\section*{Step 3: minimizers for submajorization in $U_t(S_0,m)$}\label{min maj mat}

\pausa

Let $S_0\in \matpos$ and let $t\geq t_0=\tr(S_0)$.  Notice that Corollary \ref{cor nuevo} together with Theorem \ref{tutti nu} show that the sets $U_t(S_0,m)$ have minimal elements with respect to submajorization. 
We shall 
describe the geometrical structure of minimal elements in
$U_t(S_0\coma m)$  with respect to submajorization for any  $ m< d$ in terms of the geometry of $S_0$. 
We shall see that, under some mild assumptions, 
there exists a unique $S_t\in U_t(S_0\coma m)$ such that  
$\la(S_t)= \vlm( t)$ (the vector 
of Theorem \ref{tutti nu} defined in  \ref{defi del nulam}).
In order to do this  we recall a series of preliminary results and we fix some notations.

\pausa

 \noindent{\bf Interlacing inequalities}. Let $A \in \mathcal{H}(d)$ with $\la(A) \in \R^d\, ^\downarrow$ and let 
$P = P^2 = P^* \in \matpos$ be a projection with $\rk \, P = k$. 
The interlacing inequalities (see \cite{Bat}) relate the eigenvalues of $A$ with the eigenvalues of $PAP\in \mathcal{H}(d)$ as follows: 
\beq\label{inter}
\la_{d-k+i}(A) \le \la_i(PAP) \le \la_i(A)  \peso{for every} i \in \IN{k} \ .
\eeq
On the other hand, if we have the equalities 
\beq\label{inter con =}
\la_i(PAP) =  \la_i(A)  \peso{for every \ \  $i \in \IN{k}$ \ \ then}	 PA = AP  \ , 
\eeq
and that $R(P)$ has an ONB $\{h_i\}_{i\in \IN{k}}$ such that $A\,h_i=\lambda_i\,h_i$
for every  $i\in \IN{k}\,$.  
Indeed, if $Q= I-P$, then 
$\tr \, QAQ = \suml_{i=k+1}^d \la_i(A)$. The interlacing inequalities applied to $QAQ$ 
imply that $$\la_{k+j}(A) \le \la_j(QAQ) \peso{for} j \in \IN{d-k}
\implies  \la_j(QAQ) =  \la_{k+j}(A) \peso{for} j \in \IN{d-k}\,.$$ 
Taking Frobenius norms, we get that 
$$
\|A\|_{_2}^2 = \sum_{i=1}^d \la_i(A)^2 = \|PAP\|_{_2}^2 + \|QAQ\|_{_2}^2  
\implies PAQ = QAP = 0 \ ,
$$
so that $A = PAP +QAQ$. The Ky-Fan inequalities (see \cite{Bat}) assure that 
\beq\label{KF}
\sum_{i=1}^k \la_i(A) = \max \, \Big\{\tr\, PAP : P \in \matpos \ , \ \ 
P= P^2=P^*  \py \rk \, P = k\, \Big\} \ .
\eeq
As before, given  an orthogonal projection $P$ with $\rk \, P = k$ such that 
\beq\label{KF con =} 
\tr \, PAP = \suml_{i=1}^k \la_i(A)
\stackrel{\eqref{inter}}{\implies} \la_i(PAP) = \la_i(A) \peso{for} 
i \in \IN{k}   \stackrel{\eqref{inter con =}}{\implies} PA = AP \ , 
\eeq
and $R(P)$ has an ONB of eigenvectors for $A$ 
associated to $\la_1(A) \coma \dots \coma \la_k(A)$. If we further assume that $\la_k(A) > \la_{k+1}(A)$ then
in both cases \eqref{inter con =} and 
\eqref{KF con =} the projection $P$ is unique, since 
the eigenvectors associated to the first $k$ eigenvalues of $A$ 
generate a unique subspace of $\C^d$.

\pausa

 \noindent{\bf Notations.} 
We fix a matrix $S\in \matpos$ with $ \lambda(S)=\la= (\lambda_1 \coma \ldots \coma \lambda_d) \in \R_{+}^d\,^\downarrow\,$. We shall also fix  
an orthonormal basis $\{h_i\}_{i\in \IN{d}}$  of $\C^d$ such that 
$$
S\,h_i=\lambda_i\,h_i\, \peso{for every} i\in \IN{d}\ .
$$ 
Any other such basis will be denoted as a 
``ONB of eigenvectors for $S\coma \la \,$".  

\begin{lem}\label{lem sobre rango}
Let $B\in \matpos$ and $r\in \IN{d-1}$ such that $\lambda(S+B)
=(\lambda_1 \coma \ldots \coma \lambda_r \coma\al )$, 
for some $\al \in \R_{+}^{d-r}\,^\downarrow$ such that 
$\al_1\leq \lambda_r\,$. 
Let $\eme_r \igdef 
\gen \{h_i: i \in \IN{r}\}$ and $P = P_{\eme_r}$. Then 
$$PB = BP = PBP=0  \ .
$$ 
\end{lem}
\begin{proof}
Since $\rk \, P=r$ and $\tr(PSP)= \suml_{i=1}^r\lambda_i\,$,  
then the Ky Fan theorem \eqref{KF} assures that 
$$
0\leq \tr(PBP)=\tr(P(S+B)P)-\tr(PSP)\leq \sum_{i=1}^r\lambda_i(S+B)-\sum_{i=1}^r\lambda_i=0\ .
$$ 
Since $B\ge 0$, we have that 
 $\tr(PBP)=0  \implies PBP=0\implies BP= PB = 0$. 
\end{proof}

\begin{pro}\label{lem sobre rango bis}
Let $r\in \IN{d-1}$, then for each $c \in[\la_{r+1}\coma \lambda_r]$ there is a unique $B\in \matpos$ such that $\lambda(S+B)
=(\lambda_1 \coma \ldots \coma \lambda_r \coma c \, \uno_{d-r})$. 
Moreover,  it is given by  \beq\label{eqsugdem}B=  \suml_{i=1}^{d-r} (c-\la _{r+i}) \, 
h_{r+i} \otimes h_{r+i} \peso{ and } S+B=\suml_{i=1}^r \lambda_i \cdot h_{i} \otimes h_{i}+ c\cdot \suml_{i=r+1}^d h_{i} \otimes h_{i}\,.\eeq
\end{pro}
\begin{proof}
Let $\eme_r \igdef \gen \{h_i: i \in \IN{r}\}$ and $P = P_{\eme_r}$. Suppose that $B\in \matpos$ is such that $\lambda(S+B)
=(\lambda_1 \coma \ldots \coma \lambda_r \coma c \, \uno_{d-r})$. Then,  
by Lemma \ref{lem sobre rango}, $BP = PB=0$. Hence 
$$
P(S+B)P=  (S+B)P=SP = \suml_{i=1}^{r} \la _{i} \, h_{i} \otimes h_{i}   
\stackrel{\rm Eq. \,\eqref{KF con =}}{\implies} (S+B)Q=c\,Q \ , 
$$
where $Q=I-P$. Hence $B= BQ= c\,Q -S\,Q = \suml_{i=1}^{d-r} (c-\la _{r+i}) \, 
h_{r+i} \otimes h_{r+i}\,$. 
\end{proof} 

\begin{rem} In Lemma \ref{lem sobre rango}, we allow the case where
$\la_r=\la_{r+1}= \al_1\,$. In this case
we could change $h_r $ by $h_{r+1}$ (or any other eigenvector for $\la_r$) 
as a generator for $\eme_r\,$. The proof of the Lemma assures that we get another 
projector $P'$ which also satisfies that $BP'=0$. 

\pausa
Similarly, in Proposition \ref{lem sobre rango bis} we allow the case where
$\la_r=\la_{r+1}= c\,$.  By the previous comments, the projection $P$  in the proof of 
Proposition \ref{lem sobre rango bis} is not unique. Nevertheless, in this case 
 the positive perturbation $B$ is unique, because we have that $\rk \, B < d-m$ (this follows from the 
fact that $(c- \la_{r+1})\, h_{r+1} \otimes h_{r+1} = 0$). In fact $B = 
c\,Q -S\,Q $, where $Q$ is the orthogonal projector onto the sum of the eigenspaces of $S$ for the
eigenvalues $\la_i < c$. 
\EOE
\end{rem}

\begin{lem}\label{lem sobre rango tris}
Let $m\in \IN{d-1}$ and $B\in \matpos$ with $\rk \, B \le d-m$. 
Assume that 
$$\lambda(S+B)
=(c\uno _{d-m} \coma \lambda_1 \coma \ldots \coma \lambda_m  ) \ ,
$$ 
for some $c \ge  \lambda_1\,$. Then there exists 
an ONB 
$\{v_i\}_{i\in \IN{d}}\,$  of eigenvectors for 
$S\coma \la$ such that 
\beq\label{EL B}
B=  \suml_{i=1}^{d-m} (c-\la _{m+i}) \, 
v_{m+i} \otimes v_{m+i} \peso{ so that } S+B
=\suml_{i=1}^m \lambda_i \cdot v_{i} \otimes v_{i}+ c\cdot \suml_{i=m+1}^d v_{i} \otimes v_{i}\ .
\eeq
If we assume further that $\la_m >\la_{m+1} $ then $B$ is unique, and Eq. \eqref{EL B} holds
for any ONB of eigenvectors for $S\coma \la \,$.
\end{lem}
\begin{proof}
Note that, since $\rk\, B\le d-m$, then   
\beq\label{rank justo}
\suml_{i=1}^{d-m} \la_i (B) = \tr B = \tr (B+S)-\tr S 
= c\, (d-m) -\suml_{j=m+1}^{d} \la_{j} \ .
\eeq
Take a subspace $\eme\inc \cene$ such that $R(B) \inc \eme$ and $\dim \eme = d-m$. 
Denote by $Q=  P_{\eme}\,$.  
Then $QBQ = B$, and the  Ky-Fan inequalities \eqref{KF} for $S+B$ assure that 
$$
\barr{rl}
\tr(QSQ) & =  \ \tr(Q(S+B)Q)-\tr B \\&\\
&   \le  \ \suml_{i=1}^{d-m}\lambda_i(S+B)- \tr B  
= c\, (d-m)-\tr B \ \stackrel{\eqref{rank justo}}{=}\ 
\suml_{j=m+1}^{d} \la_j \ . \earr
$$ The equality in Ky-Fan inequalities (for $-S$) force that 
$\eme =  \gen\{v_{m+1}\coma  \dots \coma v_d\}$, for some 
ONB $\{v_i\}_{i\in \IN{d}}$ of eigenvectors for $S\coma \la \,$ 
(see the remark following 
Eq. \eqref{KF con =}\,). Thus, we get that $Q\, S = S\,Q = 
\suml_{i=1}^{d-m} \la _{m+i} \, v_{m+i} \otimes v_{m+i}\,$. 
Since $R(B) \inc \eme$ then  $P \igdef I-Q \leq  P_{\ker B}\,$, and 
$$
B\, P = 0 \implies 
P(S+B)P = S\,P = \suml_{i=1}^{m} \la _{i} \, v_{i} \otimes v_{i} 
\stackrel{Eq. \eqref{KF con =}}{\implies}  (S+B)\,Q = c\, Q \ .
$$
Therefore we can now compute 
\begin{equation}\label{despeje B} B = B\,Q = (S+B)Q  -SQ 
=  \suml_{i=1}^{d-m} (c-\la _{m+i}) \, 
v_{m+i} \otimes v_{m+i}\ .\end{equation} 
Finally, if we further assume that $\lambda_m>\lambda_{m+1}$ 
then the subspace  $\eme =  \gen\{v_{m+1}\coma  \dots \coma v_d\}$ is 
independent of the choice of the ONB of eigenvectors for $S\coma \la \,$. 
Thus, in this case $B$ is uniquely determined by \eqref{despeje B}.
\end{proof}

\begin{pro} \label{lem sobre rango quatris}
Let $m\in \IN{d-1}$ and $B\in \matpos$ with $\rk \, B \le d-m$. 
Let $c\in \R$ such that $ \la_{r+1} \le c< \la_r\,$, for some 
$r<m$. Assume that 
$$
\lambda(S+B) =  \Big( \la_1\coma \dots \coma \la_r \coma c \, \uno_{d-m}\coma \la_{r+1}
\coma \dots \coma \la_m\Big) \ .
$$
Then there exists 
an ONB 
$\{v_i\}_{i\in \IN{d}}\,$  of eigenvectors for 
$S\coma \la$ such that 
$$B=  \suml_{i=1}^{d-m} (c-\la _{m+i}) \, 
v_{m+i} \otimes v_{m+i} \peso{ so that } S+B
=\suml_{i=1}^m \lambda_i \cdot v_{i} \otimes v_{i}+ 
c\cdot \suml_{i=m+1}^d v_{i} \otimes v_{i}\ .
$$ 
If we further assume that $\la_m >\la_{m+1} $ then $B$ is unique.
\end{pro}
\begin{proof}
Consider the subspace $\eme_r = \gen\{h_1\coma \dots \coma h_r\}$ and $P = P_{\eme_r}\,$. 
By Lemma \ref{lem sobre rango}, we know that $P\, B = B\,P = 0$. 
Let $S_1 = S\big|_{\eme_r\orto} $ and $ B_1 = B\big|_{\eme_r\orto} \ (=B) $ 
considered as 
operators in $L(\eme_r\orto)$. Then $S_1$ and $B_1$ are in the 
conditions of 
Lemma \ref{lem sobre rango tris}, so that there exists 
an ONB 
$\{w_i\}_{i\in \IN{d-r}}$ of $\eme_r\orto$  of eigenvectors for 
$S_1 \coma (\la_{r+1}\coma \dots \coma \la_d)$ such that 
$$ 
B= B_1=  \suml_{i=1}^{d-m} (c-\la _{m+i}) \, 
w_{m+i} \otimes w_{m+i}\ .
$$ 
Finally, let $\{v_i\}_{i\in \IN{d}}$ be given by $v_i=h_i$ for 
$1\leq i\leq r$ and $v_{r+i}=w_i$ for $r+1\leq i\leq d$. Then $\{v_i\}_{i\in \IN{d}}$ has the desired properties. 
Notice that if we further assume that $\lambda_m>\lambda_{m+1}$ then Lemma \ref{lem sobre rango tris} implies that $B_1$ is unique and therefore $B$ is unique, too.
\end{proof}

\begin{rem}\label{rem param 1}
With the notations of Lemma \ref{lem sobre rango tris} assume that 
$\lambda_m=\lambda_{m+1}\,$. In this case  $B$ is not uniquely determined. 
Next we obtain a parametrization of the set of all operators $B\in \matpos$ such that 
$\lambda(S+B) =(c\uno _{d-m} \coma \lambda_1 \coma \ldots \coma \lambda_m  )$. 
Consider $p=(d-m)-\#\{i:\ \lambda_i<\lambda_{m+1}\}$ and notice that in this case 
we have that $1\leq p< \#\{i:\ \lambda_i=\lambda_{m+1}\}=\dim\ker(S-\lambda_{m+1}\,I)$. 
Then, for every  $B\in\matpos$ as above there corresponds a subspace 
$\ene=\gen\{h_i: \ m+1\leq i\leq m+p\}\subset \ker(S-\lambda_m\,I)$ with 
$\dim\ene=p$ such that 
\begin{equation}\label{eq param solB} 
B=(c-\lambda_m)\ P_\ene 
+ \sum_{i=p+1}^{d-m} (c-\lambda_{m+i})\,h_{m+i} \otimes h_{m+i} \ .
\end{equation}
Conversely, for every subspace $\ene\subset \ker(S-\lambda_m\,I)$ with $\dim\ene=p$ 
then the operator $B\in \matpos$ given by \eqref{eq param solB} satisfies  that 
$\lambda(S+B) =(c\uno _{d-m} \coma \lambda_1 \coma \ldots \coma \lambda_m )$. 
Since the previous map $B\mapsto P_\ene$ is bijective, we see that the set of 
all such operators $B$ is parametrized by the set of projections $P_\ene$ such 
that $\ene\subset \ker(S-\lambda_m\,I)$ is a
$p$-dimensional subspace. Moreover, this map is actually an homeomorphism 
between these sets, with their usual metric structures. 

\pausa
Finally, if we let $k=\#\{ i:\ \lambda_i>\lambda_m\}$ then 
the set of operators $S+B$ such that $B\in \matpos$ with $\rk B\leq m-d$  and such 
that $\lambda(S+B)=(c\uno _{d-m} \coma \lambda_1 \coma \ldots \coma \lambda_m  )$ is given by 
$$ 
S+B=\suml_{i=1}^k \lambda_i\cdot h_i\otimes h_i + \lambda_{m} \cdot P_{\ene\,'}
+ c \cdot ( P_{\ene}+ \suml_{i=p+1}^{d-m} h_i\otimes h_i ) \ ,
$$ 
where $\ene\subset \ker(S-\lambda_m\,I)$ is a subspace with $\dim\ene=p$ and 
$\ene\,'=\ker(S-\lambda_{m+1}\,I)\cap \ene\orto$.

\pausa
As a consequence of the proof of Proposition \ref{lem sobre rango quatris}, 
we have a similar description of the operators $B$ of its statement. 
\EOE
\end{rem}

\subsection*{Proofs of the main results}

\pausa

\proof[Proof of Theorem \ref{el St general}]
It is a consequence of Corollary \ref{cor nuevo},  Theorem \ref{tutti nu}, 
and the results of this section (Lemma 
\ref{lem sobre rango tris}
and Propositions \ref{lem sobre rango bis}, 
\ref{lem sobre rango quatris}). The arrow (b) $\implies $ (a) in 
Item 2 follows 
by Definition \ref{defi del nulam} and the fact that both matrices 
$S_{0}$ and $B$ are diagonal on the same basis (as, for example, in Eq. \eqref{eqsugdem}).\QED

\proof[Proof of Proposition \ref{pdefi del nulam}] 
It is a consequence of Corollary \ref{cor nuevo}, Definition \ref{defi del nulam}, Lemma \ref{las propos del nulam} and Theorem \ref{tutti nu}.
\QED

\medskip

\noindent{\bf Acknowledgment. } The authors would like to thank the reviewers of the manuscript for several useful suggestions that improved the exposition of the results contained herein. 

\fontsize {8}{9}\selectfont

\noindent Pedro Massey\\
FCE - Universidad Nacional de La Plata and IAM - CONICET\\
massey@mate.unlp.edu.ar

\medskip

\noindent  Mariano Ruiz\\
FCE - Universidad Nacional de La Plata and IAM - CONICET\\
mruiz@mate.unlp.edu.ar
 
\medskip

\noindent Demetrio Stojanoff\\
FCE - Universidad Nacional de La Plata and IAM - CONICET\\
demetrio@mate.unlp.edu.ar
 
\end{document}